\documentclass[12pt]{amsart}
\usepackage{graphicx}
\usepackage{amssymb}
\usepackage{amsmath}
\usepackage{amsthm}
\usepackage{mathtools}
\usepackage{tikz}
\usepackage{psfrag}

\newtheorem{thm}{Theorem}[section]
\newtheorem{lem}[thm]{Lemma}
\newtheorem{pro}[thm]{Proposition}
\newtheorem{coro}[thm]{Corollary}

\theoremstyle{definition}
\newtheorem{defn}{Definition}[section]
\newtheorem{remk}[defn]{Remark}

\newcommand{\N}{\mathbf M^{n\times n}}

\newcommand{\n}{\mathbf n}

\newcommand{\B}{\mathbf B}
\newcommand{\R}{\mathbf R}
\newcommand{\dist}{\operatorname{dist}}

\newcommand{\X}{\mathcal X}
\newcommand{\Y}{\mathcal Y}
\newcommand{\G}{\mathcal G}
\newcommand{\tr}{\operatorname{tr}}
\newcommand{\rank}{\operatorname{rank}}
\newcommand{\dv}{\operatorname{div}}

\newcommand{\wcon}{\rightharpoonup}

\makeatletter
\numberwithin{equation}{section}
\makeatother

\begin{document}

\title[Convex integration and the  Perona-Malik  equation]{Convex integration and infinitely many weak solutions to the  Perona-Malik  equation \\  in all dimensions}
\author{Seonghak Kim and Baisheng Yan}
\address{Department of Mathematics\\ Michigan State
University\\ East Lansing, MI 48824, U.S.A.}
\email{kimseo14@msu.edu}
\email{yan@math.msu.edu}

\begin{abstract}
We prove  that for all smooth nonconstant   initial data the initial-Neumann boundary value problem for the Perona-Malik equation  in image processing   possesses  infinitely many Lipschitz  weak solutions on  smooth bounded convex  domains in all dimensions. Such   existence results have not been known except for the one-dimensional problems. Our  approach  is motivated by reformulating the Perona-Malik equation as a nonhomogeneous partial differential inclusion with  linear constraint and uncontrollable  components of gradient. We establish  a general existence result by a suitable  Baire's category method under a pivotal  density hypothesis. We finally fulfill this density hypothesis  by  convex integration based on certain approximations  from   an explicit formula of  lamination convex hull of some matrix set involved.  
\end{abstract}
\maketitle

\section{Introduction}

In this paper, we study  the initial and Neumann boundary value problem:
\begin{equation}\label{ib-PM}
\begin{cases} u_t =\dv \left (\frac{Du}{1+|Du|^2}\right ) &\mbox{in $\Omega\times (0,T),$}\\
{\partial u}/{\partial \n}=0 & \mbox{on $\partial \Omega\times (0,T),$}\\
u(x,0)=u_0(x), &x\in \Omega,
\end{cases}
\end{equation}
where $\Omega\subset \R^n$  is a smooth bounded convex domain, $T>0$ is a  given number, $u=u(x,t)$ is the unknown function with $u_t$ denoting its time-derivative and $Du=(u_{x_1},\cdots,u_{x_n})$ its spatial gradient, $\n$ is outer unit normal on $\partial\Omega$, and $u_0(x)$ is a  given smooth  function satisfying 
\begin{equation}\label{ib-PM-1}
Du_0\not\equiv 0 \;\; \mbox{in $\Omega$,}\quad {\partial u_0}/{\partial \n}=0\;\; \mbox{on $\partial \Omega.$}
\end{equation}

Problem (\ref{ib-PM}), especially when $n=2$, is a famous {\em Perona-Malik  model}   in image processing introduced by Perona and Malik  \cite{PM} for denoising and edge enhancement of a computer vision.  In this model, $u(x,t)$ represents an improved version of the initial gray level $u_0(x)$ of a noisy picture.
The anisotropic diffusion $\dv(\frac{Du}{1+|Du|^2})$ is forward parabolic in the \emph{subcritical} region where $|Du|<1$ and  backward parabolic in the  \emph{supercritical} region where $|Du|>1.$ 

The expectation of the Perona-Malik model is that disturbances with small gradient in the subcritical region will be smoothed out by the forward parabolic diffusion, while sharp edges corresponding to  large gradient  in the supercritical region  will be enhanced by the backward parabolic equation. Such  expected phenomenology has been implemented and observed in some numerical experiments, showing the stability and effectiveness of the model. On the other hand, many analytical works have shown that the model is highly ill-posed when the initial datum $u_0$ is    \emph{transcritical} in $\Omega$; namely,  there are subregions in $\Omega$ where  $|Du_0|<1$ and  where $|Du_0|>1$, respectively. For transcritical initial data, due to the backward parabolicity,  even a proper  notion and the existence of well-posed solutions to (\ref{ib-PM}) have remained largely unsettled.  Most  analytical works have   focused on the study of  singular perturbations, Young measure solutions, numerical scheme analyses, and  examples and properties  of certain classical solutions; see, e.g.,   \cite{BF, CZ, Es, GG1, GG2, KK, Ky}.

The present paper addresses the analytical issue concerning the existence of certain exact weak solutions  to problem (\ref{ib-PM}). Let $\Omega_T=\Omega\times (0,T)$. We say that a Lipschitz function $u\in W^{1,\infty}(\Omega_T)$  is   a  {\em weak solution} to (\ref{ib-PM}) provided for all $\zeta\in C^\infty(\bar\Omega_T)$ and  $s\in [0,T]$,
\begin{equation}\label{bdry-0}
 \int_\Omega u(x,s)\zeta(x,s)dx  + \int_0^s\int_\Omega \left (-u\zeta_t +\sigma(D u)\cdot D \zeta\right)dxdt=  \int_\Omega u_0(x)\zeta(x,0)dx,
\end{equation}
where $\sigma(p)=\frac{p}{1+|p|^2}$ ($p\in \R^n$) is the {\em Perona-Malik function.} 
The first  existence result on  such weak solutions  was established by {\sc K.\,Zhang} \cite{Zh} for the one-dimensional problem, whose pivotal idea is  to reformulate the one-dimensional Perona-Malik equation as a differential inclusion with linear constraint and then  prove the existence using a modified method of convex integration following the ideas of \cite{Ki,MSv1}.  Based on a similar approach of differential inclusion, we have  recently proved  in \cite{KY} that for all dimensions $n$ if the domain $\Omega$ is a ball and the nonconstant initial function $u_0$ is smooth and radially symmetric then (\ref{ib-PM}) admits infinitely many radially symmetric Lipschitz weak solutions.

The main purpose of this paper is to extend the results  of \cite{KY, Zh} to  problem (\ref{ib-PM}) on all $n$-dimensional smooth convex domains  for all  nonconstant smooth  initial data.

Our main result of the paper  is the following theorem.

\begin{thm}\label{thm:main-1}
Let $\Omega\subset\R^n$ be a bounded convex domain with  $\partial\Omega$ of $C^{2+\alpha}$ and let $u_0\in C^{2+\alpha} (\bar\Omega)$   satisfy   $(\ref{ib-PM-1})$  for some constant  $0<\alpha<1.$ 
Then $(\ref{ib-PM})$  possesses infinitely many weak solutions. Moreover, if $\|Du_0\|_{L^\infty(\Omega)}\ge 1$ and $\lambda>0$, then these weak solutions $u$  will  satisfy  the \emph{almost  gradient   maximum principle:}
 \[
\|Du\|_{L^\infty(\Omega_T)}\leq \|Du_0\|_{L^\infty(\Omega)}+\lambda.
\]
\end{thm} 

This theorem asserts that the Perona-Malik problem (\ref{ib-PM}) admits infinitely many Lipschitz weak solutions no matter whether the  initial datum is subcritical, supercritical, or transcritical.  
 
Existence of  classical solutions to Problem (\ref{ib-PM})  depends heavily on the initial data $u_0.$  {\sc Kawohl \& Kutev} \cite{KK} showed that a classical solution exists in any dimension if $u_0$ is subcritical in $\bar{\Omega}$ (see also \cite{Ky}). Later, {\sc Gobbino} \cite{Go} showed that the problem  cannot admit a global classical solution when $n=1$ if $u_0$ is transcritical.  Recently, {\sc Ghisi \& Gobbino} \cite{GG1,GG2} have studied the existence and  properties of certain classical solutions of the Perona-Malik equation in the one-dimensional or  $n$-dimensional radially symmetric  cases with suitably chosen  initial data; their  initial values can be arbitrarily given in the subcritical region, but the values in the supercritical region must be \emph{predetermined} by the subcritical  initial values.

We remark that the convexity of  the domain is needed to guarantee a \emph{gradient maximum principle} for the classical solution to  initial-Neumann boundary value problem of a class of quasilinear  uniformly parabolic equations (see Theorem \ref{existence-gr-max} below).  This gradient maximum principle turns out to be crucial for the proof of main theorem, and   an example  in \cite[Theorem 4.1]{AR} showed that such a gradient maximum principle may fail  even for  heat equation without the convexity of the domain. However, domain convexity  seemed to be  overlooked in  \cite[Theorem 6.1]{KK}.

For  the proof  of   Theorem \ref{thm:main-1}, in what follows, we  assume the initial function $u_0$  satisfies
\begin{equation}\label{ib-PM-2}
\int_\Omega u_0(x)dx=0,
\end{equation}
since otherwise one can solve solution $\tilde u$ of  (\ref{ib-PM}) with new initial datum $\tilde{u}_0=u_0-\frac{1}{|\Omega|}\int_\Omega u_0 dx;$ then $u=\tilde u+\frac{1}{|\Omega|}\int_\Omega u_0 dx$ will solve (\ref{ib-PM}).  

Our proof   is based on a crucial generalization of  the ideas of \cite{KY, Zh, Zh1}. Let us discuss this  generalization in  some details because it exhibits several different  features from the one-dimensional setup.

Assume $u\in W^{1,\infty}(\Omega_T)$ is a weak solution to (\ref{ib-PM}) and  suppose there exists a vector function $v\in W^{1,\infty}(\Omega_T;\mathbf R^n)$ such that $\dv v=u$ and $v_t=\sigma(Du)$ a.e.\,in $\Omega_T$. Let $w=(u,v)\colon \Omega_T\to \R^{1+n}$, with  space-time Jacobian matrix denoted  by
\[
\nabla w=\begin{pmatrix}Du & u_t\\ Dv & v_t\end{pmatrix} 
\]
 as an element in the matrix space $\mathbf M^{(1+n)\times (n+1)}$.
Given $s \in\R$,  define the set $K(s)$ in  $\mathbf M^{(1+n)\times (n+1)}$ by
\begin{equation}\label{set-K}
 K(s)= \left\{ \begin{pmatrix} p & c\\ B & \sigma(p)\end{pmatrix}\,\Big | \, p\in\R^n,\,c\in\R,\,B\in\mathbf{M}^{n\times n},\,  \tr B=s \right\}.
\end{equation}
Then  $w=(u,v)$ solves the {\em nonhomogeneous partial differential inclusion}:
\[
\nabla w(x,t)\in  K(u(x,t)),   \quad \textrm{a.e. $(x,t)\in\Omega_T.$}
\]
Conversely, suppose we have found a function $\Phi=(u^*,v^*),$ where  $u^*\in W^{1,\infty}(\Omega_T)$ and $v^*\in W^{1,\infty}(\Omega_T;\R^n),$ such that
\begin{equation}\label{bdry-1}
\begin{cases} u^*(x,0)=u_0(x) \; (x\in\Omega),\\
\dv v^*=u^*\;\;\textrm{a.e. in $\Omega_T$}, \\
v^*(\cdot,t) \cdot \mathbf n|_{\partial\Omega} =0 \; \; \forall\; t\in [0,T].\end{cases}
\end{equation}
 Assume $w=(u,v)\in W^{1,\infty}(\Omega_T;\R^{1+n})$ solves  the Dirichlet problem of nonhomogeneous differential inclusion:
\begin{equation}\label{pdi-D}
\begin{cases} \nabla w(x,t)\in  K(u(x,t)),  &  \textrm{a.e. $(x,t)\in\Omega_T$,}\\
w(x,t)= \Phi(x,t), &(x,t)\in\partial \Omega_T.
\end{cases}
\end{equation}
Then it can be verified that $u$ is a weak solution to  $(\ref{ib-PM})$ (see Lemma \ref{lem32}).

The Dirichlet problem (\ref{pdi-D}) falls into the  framework of  general nonhomogeneous partial differential inclusions studied by {\sc Dacorogna \& Marcellini} \cite{DM1} using Baire's category method and by {\sc M\"uller \&  Sychev} \cite{MSy} using the convex integration method; see also \cite{Ki}.  Study of such differential inclusions  has stemmed from the successful understanding of homogeneous differential inclusions of the form $Du(x)\in K$ first encountered in the study of crystal microstructure  by {\sc Ball \& James} \cite{BJ}, {\sc Chipot \& Kinderlehrer} \cite{CK} and {\sc M\"uller \& \v Sver\'ak}  \cite{MSv1}.   Recently, the method  of  differential inclusions  has been successfully applied to other important problems ; see, e.g.,  \cite{CFG,DS,MP,MSv2,Sh,Ya1}.

We point out that  the existence result of \cite{MSy} is not applicable  to problem  (\ref{pdi-D}) even in dimension $n=1$, as has already been noticed in \cite{Zh, Zh1}.  A key condition in the main existence theorem of \cite{MSy}, when applied to (\ref{pdi-D}), would require that the boundary function $\Phi$  satisfy
\[
\nabla \Phi(x,t)\in U(u^*(x,t))\cup K(u^*(x,t)), \;\; a.e.\; (x,t)\in \Omega_T,
\]
 where $U(s)\subset \mathbf M^{(1+n)\times (n+1)} \;(s\in\R)$ are bounded sets  that are \emph{reducible to}  $K(s)$ in the sense that, for every $s_0 \in \R,\; \xi_0\in U(s_0),\; \epsilon>0$, and bounded Lipschitz domain $G\subset  \R^{n+1}$, there exist  a piecewise affine function
$w\in  W_0^{1,\infty}(G;\R^{1+n})$ and a $\delta>0$ satisfying,  for a.e.\,$z=(x,t)\in G$,
\[
\xi_0+ \nabla w(z) \in \bigcap_{|s-s_0|< \delta} U(s), \;\; 
 \int_G \dist (\xi_0 +\nabla w(z),K(s_0))\,dz <\epsilon |G|.
\]
The second condition would imply  $\tr B_0=s_0$ for each $\xi_0=\begin{pmatrix} p_0 & c_0\\ B_0 & \beta_0\end{pmatrix}\in U(s_0)$ and  $s_0\in\R$; but then $\cap_{|s-s_0|< \delta} U(s) =\emptyset,$ which makes the first condition impossible.

However, certain geometric structures  of the set $K(0)$ turn out  still  useful, especially when it comes to the relaxation of homogeneous differential inclusion $\nabla \omega(z)\in K(0)$ with $\omega=(\varphi,\psi)$.
We explicitly compute the first-order lamination set $L(K(0))$ of $K(0)$ consisting of all  $\xi\in \mathbf M^{(1+n)\times (n+1)}\setminus K(0)$ such that $\xi=\lambda\xi_1+(1-\lambda)\xi_2$ for some $\lambda\in (0,1)$ and $\xi_1,\xi_2\in K(0)$  with  $\rank(\xi_1-\xi_2)=1$. We obtain the explicit formula (see Theorem \ref{thm-rconv})
\begin{equation*}%\label{set-L-0}
L(K(0))=\left\{ \begin{pmatrix} p & c\\ B & \beta\end{pmatrix} \, \Big |\; \tr B=0, \; |\beta|^2 +(p\cdot\beta)^2 -p\cdot \beta <0\right \},
\end{equation*}
which enables us to extract enough information on the diagonal components of differential inclusion $\nabla \omega(z)\in K(0)$ and establish a relaxation result on $(D\varphi,\psi_t)$ (see Theorem \ref{main-lemma}). Although for such  relaxation we must have $\dv \psi=0$, the resulting $\varphi_t$ can be arbitrarily small; this is  important for the subsequent  handling of the linear constraint  $\dv v =u$ in problem (\ref{pdi-D}).

Another difficulty concerning  problem  (\ref{pdi-D}) is that when $n=1$, one can control $\|v_x\|_{L^{\infty}(\Omega_T)}$ in terms of $u=v_x$ (see \cite{Zh}); however, for $n\ge 2$, it is impossible to control $\|Dv\|_{L^\infty(\Omega_T)}$ in terms of  $u=\dv v.$ So, if $n\ge 2$, the space $W^{1,\infty}(\Omega_T;\R^n)$ is not suitable for the function $v$. It turns out that a suitable    space for $v$ is the space $W^{1,2}((0,T);L^2(\Omega;\R^n))$  of abstract functions (see Lemma \ref{gen-lem}); in this setting, the linear constraint $\dv v=u$ must be understood in the sense of distributions.

We design a  new approach to overcome the  lack of control on $Dv$\,:  instead of defining an admissible class for $w=(u,v)$, we define a suitable admissible class for only the functions $u\in W^{1,\infty}(\Omega_T)$, treating  $v$ as auxiliary functions. Of course, during  all the relevant constructions,   the linear constraint $\dv v=u$ must be satisfied.
In this regard, we need a linear operator $\mathcal R$ that serves as a (distributional)  right inverse of the divergence operator: $\dv \mathcal R=Id$.   By the results of \cite{BB}, such an operator may not exist as a bounded operator on certain spaces, but for our purpose, it suffices to construct such an operator $\mathcal R$ that is bounded from $L^\infty(Q\times I)$ to $L^\infty(Q\times I;\R^n)$ for the box domains $Q\times I$ in $\R^{n+1};$ this is achieved by  following  some  construction in \cite{BB}.

Finally we remark that although the result of this paper heavily relies on the explicit formula of $L(K(0))$,  the method can handle some general  forward-backward parabolic equations; however, we do not intend to discuss further results of this direction  in the present  paper.

The rest of the paper is organized as follows. In Section 2, we collect several necessary preliminary results, some of which cannot be found in the standard references.  In Section 3, we set up a new general procedure for proving  Theorem \ref{thm:main-1} under a pivotal density hypothesis of an admissible class $\mathcal U$; this setup  is suitable for a Baire's category method and simplifies some of the arguments even for the one-dimensional problem. In Section 4, as the heart of the matter for fulfilling the density hypothesis and thus proving Theorem \ref{thm:main-1}, we present the essential  geometric considerations, including an explicit  computation of the set $L(K(0))$ above and establishing a critical  relaxation property (Theorem \ref{main-lemma})  by convex integration with linear constraint. In Section  5, we  construct the suitable admissible class $\mathcal U$ after defining a specific  boundary function $\Phi=(u^*,v^*)$.  In Section 6, we fulfill the key density hypothesis for admissible class $\mathcal U$ (Theorem \ref{thm-density-1}) and finally complete the proof of Theorem \ref{thm:main-1} according to the setup of Section 3.

\section{Some preliminary results}

\subsection{Uniformly parabolic quasilinear equations} We refer to the standard references (e.g., \cite{LSU, Ln}) for general theory of parabolic equations, including some notation  concerning functions and domains of class $C^{k+\alpha}$ for integer $k\ge 0$ and  number $0<\alpha<1.$   
 
Assume $f\in C^{3}([0,\infty))$ is a function satisfying
\begin{equation}\label{1-2}
  \theta \le f(s)+2sf'(s)\le \Theta \quad \forall\;  s\ge 0,
\end{equation}
where $\Theta\ge\theta>0$ are constants. This condition is equivalent to $\theta\le (sf(s^2))'\le \Theta$ for all $s\in\R;$ hence, $\theta\le f(s)\le\Theta$ for all $s\ge 0.$ Let 
\[
A(p)=f(|p|^2)p \quad (p\in \R^n).
\]
Then we have
\[
A^i_{p_j}(p)  = f(|p|^2)\delta_{ij} + 2f'(|p|^2) p_ip_j \quad (i,j=1,2,\cdots,n;\; p\in \R^n)
\]
and hence  the \emph{uniform ellipticity condition}:
\begin{equation}\label{para-0}
\theta |q|^2 \le \sum_{i,j=1}^n A^i_{p_j}(p) q_iq_j\le \Theta |q|^2\quad \forall\; p,\;q\in\R^n.
\end{equation}

\begin{thm}\label{existence-gr-max}
Let $\Omega\subset\R^n$ be a bounded convex domain with  $\partial\Omega$ of $C^{2+\alpha}$ and  $u_0\in C^{2+\alpha} (\bar\Omega)$   satisfy  $Du_0\cdot \n=0$ on $\partial\Omega.$ 
Then the initial-Neumann boundary value problem
\begin{equation}\label{2-2-2}
  \begin{cases}
  u_t=\dv (A(Du)) & \mbox{in }\Omega_T, \\
  {\partial u}/{\partial \n}=0 & \mbox{on }\partial\Omega\times(0,T), \\
  u(x,0)=u_0(x) & \mbox{for } x\in\Omega
\end{cases}
\end{equation}
has a unique solution $u\in C^{2+\alpha,\frac{2+\alpha}{2}}(\bar\Omega_T)$. Moreover,  the gradient maximum principle holds:
\begin{equation}\label{gr-max-P}
\|Du\|_{L^\infty(\Omega_T)}=\|Du_0\|_{L^\infty(\Omega)}.
\end{equation}
\end{thm}

\begin{proof}  1. As problem (\ref{2-2-2}) is uniformly parabolic by (\ref{para-0}),  the existence of unique  classical solution $u$ in $C^{2+\alpha,\frac{2+\alpha}{2}}(\bar\Omega_T)$  follows from the standard theory; see \cite[Theorem 13.24]{Ln}.  To prove  the gradient  maximum principle (\ref{gr-max-P}),  note that, since $A\in C^3 (\mathbf R^n)$, a standard bootstrap argument based on the regularity theory of linear parabolic equations \cite{LSU, Ln} shows that the solution $u$ has all continuous partial derivatives $u_{x_ix_jx_k}$ and $u_{x_it}$   within $\Omega_T$  for $1\le i,j,k\le n.$ 

2. Let $v=|Du|^2.$ Then, within $\Omega_T,$ we compute
\begin{eqnarray*}
&\Delta v =2 Du \cdot D(\Delta u) + 2 |D^2u|^2,&\\
&u_t=\dv (A(Du))= \dv (f(v)Du)= f'(v) Dv \cdot Du + f(v) \Delta u,&\\
&\begin{split}
Du_t =& f''(v) (Dv\cdot Du) Dv + f'(v) (D^2u)  Dv \\
&+ f'(v) (D^2v) Du + f'(v) (\Delta u) Dv + f(v) D(\Delta u).\end{split}&
\end{eqnarray*}
Plugging these equations into $v_t= 2Du\cdot Du_t$, we obtain
\begin{equation}\label{para-2}
v_t- \mathcal L (v) - B\cdot Dv =-2f(|Du|^2) |D^2u|^2\le 0 \quad \mbox{in $\Omega_T$,}
\end{equation}
where operator $\mathcal L(v)$ and  coefficient $B$ are defined by
\begin{eqnarray*}
&\mathcal L(v)= f(|Du|^2)\Delta v + 2 f'(|Du|^2) Du\cdot (D^2 v)Du,&\\
&B= 2f''(v) (Dv\cdot Du)Du   + 2 f'(v) (D^2u)  Du +
 2 f'(v) (\Delta u)  Du.&
\end{eqnarray*}
We write  $\mathcal L (v)= \sum_{i,j=1}^n a_{ij} v_{x_ix_j},$
 with   coefficients $a_{ij}=a_{ij}(x,t)$ given by
\[
a_{ij}=A^i_{p_j}(Du)=f(|Du|^2)\delta_{ij} + 2 f'(|Du|^2)u_{x_i}u_{x_j} \quad  (i,j=1,\cdots,n).
\]
Note that on $\bar\Omega_T$  all eigenvalues of the matrix $(a_{ij})$ lie in $[\theta,\Theta]$. 

3. We  show
\[
 \max_{(x,t)\in\bar\Omega_T}v(x,t) = \max_{x\in \bar\Omega} v(x,0),
\]
which proves (\ref{gr-max-P}).  We prove this by contradiction. Suppose
\begin{equation}\label{claim-1}
M:=\max_{(x,t)\in\bar\Omega_T}v(x,t) > \max_{x\in \bar\Omega} v(x,0).
\end{equation}
  Let  $(x_0,t_0)\in\bar\Omega_T$ be such that  $v(x_0,t_0)=M;$ then $t_0>0.$ If $x_0\in \Omega$, then the strong maximum principle applied to (\ref{para-2})  would imply  that $v$ is constant on $\Omega_{t_0},$ which yields  $v(x,0)\equiv M$ on $\bar\Omega$, a contradiction to (\ref{claim-1}). Consequently $x_0\in \partial\Omega$ and thus $v(x_0,t_0)=M>v(x,t)$ for all $(x,t)\in \Omega_T.$   We can then apply Hopf's Lemma for parabolic equations \cite{PW}  to (\ref{para-2}) to deduce $
 \partial v(x_0,t_0) /\partial \n >0.$ However, a result of \cite[Lemma 2.1]{AR} (see also \cite[Theorem 2]{Ka}) asserts  that $\partial v/\partial\n\le 0$ on $\partial\Omega\times [0,T]$ (convexity of $\Omega$ is used and necessary here),  which gives a desired contradiction.  
 \end{proof}

\subsection{Modification of the Perona-Malik function} We need to modify the Perona-Malik function $\sigma(p)= \frac{p}{1+|p|^2}$ to obtain  a uniformly parabolic problem of type  (\ref{2-2-2}).
For this purpose, let
\[
\rho(s)=\dfrac{s}{1+s^2} \quad (s\ge 0)
\]
and, for $0<\delta<1/2,$ let $m=m_\pm(\delta)$ be the solutions of $\rho(m)=\delta$;
that is,
\begin{equation}\label{root-PM}
m_\pm(\delta)=\frac{1\pm \sqrt{1-4\delta^2}}{2\delta}.
\end{equation}

The following result can be proved in a similar way as in \cite{CZ, Zh}; we omit the proof (see Figure 1).

\begin{lem}\label{lem-modi}
Let $0<\delta<1/2$ and $1<\Lambda<m_+(\delta).$    Then there exists a function $\rho^*\in C^{3}([0,\infty))$   satisfying  that
\[
\rho^*(s)=\rho(s)\quad  \forall\; 0\le s\le m_-(\delta),
\]
\[
\rho^*(s)< \rho(s) \quad  \forall\; m_-(\delta)<s\le \Lambda,
\]
\[
\theta\le(\rho^*)'(s)\le\Theta \quad \forall\; 0\le s<\infty
\]
for some constants  $\Theta>\theta>0.$
Moreover, define $f(0)=1$ and $f(s)=\rho^*(\sqrt s)/\sqrt s$ for $s>0;$  then $f\in C^{3}([0,\infty))$ and \emph{(\ref{1-2})} is fulfilled.
\end{lem}

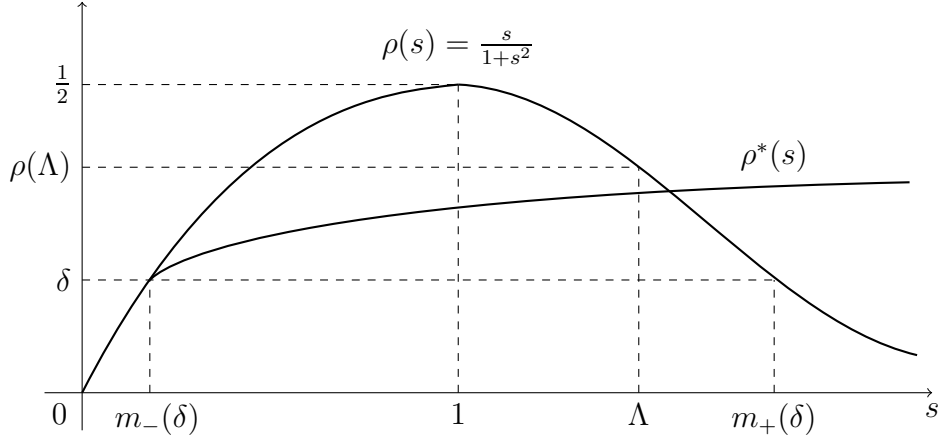
\begin{figure}[h]\label{fig1}
\begin{center}
\begin{tikzpicture}[scale = 1]
    \draw[->] (-.5,0) -- (11.3,0);
	\draw[->] (0,-.5) -- (0,5.2);

 \draw[dashed] (0,4.1)--(5,4.1);
    \draw[dashed] (5, 0)  --  (5, 4.1) ;
	\draw[thick]   (0, 0) .. controls  (2, 3.9) and (4.1,4) ..(5,4.1);
	\draw[thick]   (5, 4.1) .. controls  (7.2, 4) and (9,1) ..(11.1,0.5 );
\draw[thick]   (0.9,1.5) .. controls  (1,1.6) and (2,2.6) ..(11,2.8);
	\draw (11.3,0) node[below] {$s$};
    \draw (-0.3,0) node[below] {{$0$}};
    \draw (5, 4.2) node[above] {$\rho(s)=\frac{s}{1+s^2}$};
    \draw (0, 4.1) node[left] {$\frac12$};
   \draw (0, 1.5) node[left] {$\delta$};
   \draw (0, 3) node[left] {$\rho(\Lambda)$};
 \draw[dashed] (0,1.5)--(9.2, 1.5);
 \draw[dashed] (0,3)--(7.4, 3);
 \draw[dashed] (7.4,0)--(7.4, 3);
\draw[dashed] (9.2,0)--(9.2, 1.5);
 \draw (9.2,0) node[below] {$m_+(\delta)$};
\draw (9.2,3.5) node[below] {$\rho^*(s)$};
 \draw (7.4,0) node[below] {$\Lambda$};
    \draw (5, 0) node[below] {$1$};
 \draw[dashed] (0.9,0)--(0.9, 1.5);
 \draw (1, 0) node[below] {$m_-(\delta)$};
    \end{tikzpicture}
\end{center}
\caption{The graphs of function  $\rho(s)$ and the modified function $\rho^*(s)$ in Lemma \ref{lem-modi}.}
\end{figure}

\subsection{Right inverse of the divergence operator} To deal with  the linear constraint $\dv v=u$, we follow an argument of \cite[Lemma 4]{BB} to construct  a right inverse $\mathcal R$ of the divergence operator: $\dv \mathcal R=Id$ (in the sense of distributions in $\Omega_T$). For the purpose of this paper, the construction of $\mathcal R$ is restricted to the {\em box} domains, by which we mean  domains given by $Q=J_1\times J_2\times \cdots\times J_n$, where $J_i=(a_i,b_i)\subset\R$ is a finite open  interval.

Given such a box $Q$, we define a linear operator $\mathcal R_n\colon L^\infty(Q)\to L^\infty(Q;\mathbf R^n)$ inductively on dimension $n$.
If $n=1$, for $u\in L^\infty(J_1)$, we define $v=\mathcal R_1 u$ by
\[
v(x_1)=\int_{a_1}^{x_1} u(s)ds\quad (x_1\in J_1).
\]
Assume $n=2$. Let $u\in L^\infty(J_1\times J_2).$ Set
$\tilde u(x_1)=\int_{a_2}^{b_2} u(x_1,s)\,ds$ for $x_1\in J_1.$ Then $\tilde u\in L^\infty(J_1).$ Let $\tilde v=\mathcal R_1\tilde u$; that is,
\[
\tilde v(x_1)=\int_{a_1}^{x_1}\tilde u(s)ds=\int_{a_1}^{x_1}\int_{a_2}^{b_2} u(s,\tau)\,d\tau ds \quad (x_1\in J_1).
\]
 Let $\rho_2\in C^\infty_c(a_2,b_2)$ be such that $0\le \rho_2(s)\le\frac{C_0}{b_2-a_2}$ and $\int_{a_2}^{b_2} \rho_2(s) ds=1.$
Define $v=\mathcal R_2 u\in L^\infty(J_1\times J_2;\mathbf R^2)$  by $v=(v^1,v^2)$ with $v^1 (x_1,x_2)=\rho_2(x_2) \tilde v(x_1)$ and
\[
v^2(x_1,x_2) =\int_{a_2}^{x_2}u(x_1,s)ds -\tilde u(x_1) \int_{a_2}^{x_2} \rho_2(s) ds.
\]
Note that if $u\in W^{1,\infty}(J_1\times J_2)$ then $\tilde u\in W^{1,\infty}(J_1)$; hence $v=\mathcal R_2 u\in W^{1,\infty}(J_1\times J_2;\mathbf R^2)$ and $\dv v = u$ a.e.\,in $J_1\times J_2.$ Moreover, if $u\in C^1(\overline{J_1\times J_2})$ then $v$ is in $C^1(\overline{J_1\times J_2};\mathbf R^2).$

Assume that we have defined the operator $\mathcal R_{n-1}$. Let $u\in L^\infty(Q)$ with $Q=J_1\times J_2\times \cdots\times J_n$  and $x=(x',x_n)\in Q$, where $x'\in Q'=J_1\times\cdots\times J_{n-1}$ and $x_n\in J_n.$ Set
$\tilde u(x')=\int_{a_n}^{b_n} u(x',s)\,ds$ for $x'\in Q'.$
Then $\tilde u\in L^{\infty}(Q').$ By  the  assumption, $\tilde v =\mathcal R_{n-1}\tilde u\in L^\infty(Q';\mathbf R^{n-1})$ is defined.  Write $\tilde v(x')=(Z^1(x'),\cdots,Z^{n-1}(x')),$ and let $\rho_n\in C^\infty_c(a_n,b_n)$ be a  function satisfying
$0\le \rho_n(s)\le\frac{C_0}{b_n-a_n}$ and $\int_{a_n}^{b_n} \rho_n(s) ds=1.$
Define $v=\mathcal R_n u\in L^\infty(Q;\mathbf R^n)$ as follows.  For $x=(x',x_n)\in Q,$  $v(x)=(v^1(x),v^2(x),\cdots,v^n(x))$ is defined by
\[
\begin{split}
& v^k(x',x_n)=\rho_n(x_n) Z^k(x') \quad  (k=1,2,\cdots, n-1),\\
&v^n (x',x_n)=\int_{a_n}^{x_n}u(x',s)ds -\tilde u(x') \int_{a_n}^{x_n} \rho_n(s) ds.
\end{split}
\]
Then  $\mathcal R_n \colon L^\infty(Q)\to L^\infty(Q;\mathbf  R^n)$  is a well-defined linear operator; moreover,
\begin{equation}\label{div-0}
 \|\mathcal R_n  u\|_{L^\infty(Q)}  \le C_n \,(|J_1|+\cdots+|J_n|) \|u\|_{L^\infty(Q)},
\end{equation}
where $C_n>0$ is a constant depending only on $n.$

As in the case $n=2$, we  see that if $u\in W^{1,\infty}(Q)$ then $v=\mathcal R_n u\in W^{1,\infty}(Q;\mathbf R^n)$ and
$\dv v = u$ a.e.\,in $Q$. Also, if $u\in C^1(\bar Q)$ then $v=\mathcal R_n u$ is in $C^1(\bar Q;\mathbf R^n).$ Moreover,   if $u\in W^{1,\infty}_0(Q)$ satisfies  $\int_Q u(x)dx=0$, then one can easily show that $v=\mathcal R_n u\in W_0^{1,\infty}(Q;\mathbf R^n).$

Let  $I$ be a finite open interval in $\mathbf R$. We now  extend the operator $\mathcal R_n$ to an operator $\mathcal R$ on $L^\infty(Q\times I)$ by defining,  for a.e.\,$(x,t)\in Q\times I$,
\begin{equation}\label{def-R}
(\mathcal R u)(x,t)=(\mathcal R_n u(\cdot,t))(x) \quad \forall\; u\in L^\infty(Q\times I).
\end{equation}
 Then $\mathcal R \colon  L^\infty(Q\times I)\to L^\infty(Q\times I;\mathbf R^n)$ is a bounded linear operator.

We have the following result.

\begin{thm}\label{div-inv} Let $u\in W^{1,\infty}_0(Q\times I)$ satisfy  $\int_{Q}u(x,t)\,dx=0$ for all $t\in I$. Then $v=\mathcal R u\in W^{1,\infty}_0(Q\times I;\mathbf R^n)$,  $\dv v=u$ a.e.\,in $Q\times I$, and
\begin{equation}\label{div-1}
\|v_t \|_{L^\infty(Q\times I)}  \le C_n \,(|J_1|+\cdots+|J_n|) \|u_t\|_{L^\infty(Q\times I)},
\end{equation}
where $Q=J_1\times\cdots\times J_n$ and $C_n$ is the same constant as in $(\ref{div-0})$.  Moreover, if $u\in C^1(\overline{Q\times I})$ then $v =\mathcal R u \in C^1(\overline{Q\times I};\mathbf R^{n}).$
\end{thm}
\begin{proof} Given $u\in  W^{1,\infty}_0(Q\times I)$, let $v=\mathcal R u.$ We easily verify that $v$ is Lipschitz continuous in $t$ and hence  $v_t$ exists. It also follows that $v_t=\mathcal R (u_t).$ Clearly, if $\int_Q u(x,t)dx=0$ then $v(x,t)=0$ whenever $t\in \partial I$ or $x\in \partial Q$. This proves $v\in W^{1,\infty}_0(Q\times I;\mathbf R^{n})$ and  the estimate (\ref{div-1}) follows from (\ref{div-0}). Finally, from the definition of $\mathcal R u$, we see that if $u\in C^1(\overline{Q\times I})$ then $v =\mathcal R u \in C^1(\overline{Q\times I};\mathbf R^{n}).$
\end{proof}

 \section{General  setup for existence}
In this section we set up the general procedure for proving our main theorem, Theorem \ref{thm:main-1}.

\subsection{Sufficient conditions for weak solutions} Since our setup differs from the usual formulation of  differential inclusions, we first prove the next two results to clarify some relevant issues, which are elementary but not too obvious.

\begin{lem} \label{gen-lem} Suppose $u\in W^{1,\infty}(\Omega_T)$ is such that $u(x,0)=u_0(x)$ $(x\in\Omega)$,   there exists a vector function $v\in  W^{1,2}((0,T);L^2(\Omega;\mathbf R^n))$ with weak time-derivative $v_t$  satisfying $v_t=\sigma(Du)$ a.e.\,in $\Omega_T$, and  for each $\zeta\in C^\infty(\bar\Omega_T)$ and    $t\in [0,T],$
\begin{equation}\label{div-v=u}
\int_\Omega v(x,t)\cdot D\zeta(x,t)\,dx=-\int_\Omega u(x,t)\zeta(x,t)\,dx.
\end{equation}
Then $u$ is a weak solution to $(\ref{ib-PM})$.
\end{lem}

\begin{proof}  To verify (\ref{bdry-0}), given any $\zeta\in C^\infty(\bar\Omega_T)$, let
\[
g(t)=\int_\Omega u(x,t)\zeta(x,t)dx,\quad h(t)=\int_\Omega u(x,t)\zeta_t(x,t)dx \quad (0\le t\le T).
\]
 Then for each $\psi\in C^\infty_c(0,T),$ by (\ref{div-v=u}),
\begin{eqnarray*} 
\int_0^T \psi_t(t) g(t) \,dt  = -\int_0^T\int_\Omega \psi_t(t) v(x,t)\cdot D\zeta(x,t)\,dxdt,\\
\int_0^T \psi(t) h(t) \,dt = -\int_0^T\int_\Omega \psi(t) v(x,t)\cdot D\zeta_t(x,t)\,dxdt.
\end{eqnarray*}
Since $v\in W^{1,2}((0,T);L^2(\Omega;\mathbf R^n))$ and $v_t=\sigma(Du)$, one has
\[
\int_0^T\int_\Omega (\psi(t) D\zeta(x,t))_t\cdot v(x,t) dxdt=  - \int_0^T\int_\Omega  \psi (t)   \sigma(Du(x,t))\cdot D\zeta(x,t)\,dxdt.
\]
Now as  $(\psi  D\zeta )_t=\psi_t D\zeta + \psi D\zeta_t$, combining  previous equations, we have
\[
\int_0^T \psi_t(t) g(t) \,dt= \int_0^T \psi(t) \left (-h(t)+\int_\Omega    \sigma(Du(x,t))\cdot D\zeta(x,t)\,dx\right )dt;
\]
this proves that $g$ is weakly differentiable in $(0,T)$ with weak derivative  
\[
g'(t)=h(t)-\int_\Omega \sigma(Du(x,t))\cdot D\zeta(x,t)\,dx \; \; \mbox{ a.e. $t\in(0,T)$.}
\]
From this, upon integrating,  (\ref{bdry-0}) follows for all $s\in [0,T].$
\end{proof}

Let $\Phi=(u^*, v^*)\in W^{1,\infty}(\Omega_T;\mathbf R^{1+n})$  satisfy  (\ref{bdry-1}),  and let $W^{1,\infty}_{u^*}(\Omega_T)$ and $W^{1,\infty}_{v^*}(\Omega_T;\mathbf R^n)$ denote the usual Dirichlet classes with boundary traces $u^*$,  $v^*$, respectively.

Let $\mathcal U$ be   {\em some} nonempty and bounded subset of $ W^{1,\infty}_{u^*}(\Omega_T)$ such that for each $u\in \mathcal U$,  there exists a vector function $v\in W_{v^*}^{1,\infty}(\Omega_T; \mathbf R^n)$ satisfying  $\dv v=u$ a.e.\,in $\Omega_T$ and  $\|v_t\|_{L^\infty(\Omega_T)}\le 1/2.$ Any such set $\mathcal U$ is called an \emph{admissible class}.

Given any $\epsilon>0$, define $\mathcal U_\epsilon$ to be the set of  $u\in\mathcal U$ such that  there exists a vector function $v\in W_{v^*}^{1,\infty}(\Omega_T; \mathbf R^n)$ satisfying  $\dv v=u$ a.e.\,in $\Omega_T$,  $\|v_t\|_{L^\infty(\Omega_T)}\le 1/2$,   and
\[
\int_{\Omega_T} |v_t(x,t)-\sigma(Du(x,t))|\,dxdt \leq\epsilon|\Omega_T|.
\]
(Note that $\mathcal U_\epsilon =\mathcal U$ for all $\epsilon\geq 1.$)

\begin{lem}\label{lem32}  Let $u\in\mathcal U.$ Then any vector function $v\in W_{v^*}^{1,\infty}(\Omega_T; \mathbf R^n)$ determined  above satisfies the integral identity $(\ref{div-v=u})$ for each $\zeta\in C^\infty(\bar\Omega_T)$ and $t\in [0,T]$.
\end{lem}
\begin{proof} Let $\zeta\in C^\infty(\bar\Omega_T)$ and define 
\[
h(t)=\int_\Omega \left (v(x,t)\cdot D\zeta(x,t)+u(x,t)\zeta(x,t)\right)dx.
\]
Then $h$ is continuous on $[0,T]$ and for each $\psi\in C^1[0,T]$,
\[
\begin{split} \int_0^T h(t)\psi(t)\,dt  &=\int_0^T \int_{\Omega} \psi(t)\left (v(x,t)\cdot D\zeta(x,t)+u(x,t)\zeta(x,t)\right)dxdt  
\\
 &=\int_0^T \int_{\Omega} \psi(t)\left (v(x,t)\cdot D\zeta(x,t)+\dv v (x,t)\zeta(x,t)\right)dxdt  \\
 &=\int_0^T \int_{\Omega}\dv (\zeta(x,t)\psi(t)v(x,t))\,dxdt  \\
 &= \int_0^T \int_{\Omega} \dv (\zeta(x,t)\psi(t)v^*(x,t))\,dxdt=0, 
\end{split}
\]
resulting from $v|_{\partial\Omega_T}=v^*|_{\partial\Omega_T}$ and $v^*(\cdot,t)\cdot \n|_{\partial\Omega}=0$ for all $t\in [0,T].$ Hence $h\equiv 0$ on $[0,T]$. This completes the proof.
\end{proof}

\subsection{General existence theorem by Baire's category method} We  prove a general existence theorem under  a density hypothesis.

\begin{thm}\label{thm:main}  Let $\mathcal U\subset W^{1,\infty}_{u^*}(\Omega_T)$ be an admissible class. Assume, for each $\epsilon>0$,  $\mathcal U_\epsilon$ is dense in  $\mathcal U$ under the $L^\infty$-norm.  Then,  given any $\varphi\in \mathcal U$, for each $\eta>0,$
there exists a weak solution $u\in  W_{u^*}^{1,\infty}(\Omega_T)$ to problem $(\ref{ib-PM})$ satisfying
$\|u-\varphi\|_{L^\infty(\Omega_T)}<\eta.$ Furthermore,  if $\mathcal U$ contains a function in $W^{1,\infty}_{u^*}(\Omega_T)$ that is not a weak solution to $(\ref{ib-PM}),$ then $ (\ref{ib-PM})$ admits  infinitely  many weak solutions.
\end{thm}

\begin{proof}  
1. Let $\X$ be the closure of $\mathcal U$ in the metric space $L^\infty(\Omega_T).$
Then $(\mathcal X,L^\infty)$ is a complete metric space.  By assumption, $\mathcal U_\epsilon$ is a dense subset of $\X.$ Furthermore, since $\mathcal U$ is bounded in $W_{u^*}^{1,\infty}(\Omega_T)$,  we have  $\X\subset W_{u^*}^{1,\infty}(\Omega_T)$.

2. Let  $\Y =L^1(\Omega_T;\mathbf R^{n})$. For $h>0$, define $T_h\colon  \X \to \Y$ as follows. Given any $u\in X$, write $u=u^* +w$ with $w\in W_0^{1,\infty}(\Omega_T)$ and define
\[
T_h (u) =Du^* + D(\rho_h * w),
\]
where $\rho_h(z)=h^{-N}\rho(z/h)$, with $z=(x,t)$ and $N=n+1$, is the standard mollifier on $\R^{N}$, and $\rho_h * w$ is the usual convolution on $\R^{N}$ with $w$ extended to be zero outside $\Omega_T.$
Then, for each $h>0$, the map $T_h \colon   (\X, L^\infty) \to (\Y, L^1)$ is continuous, and for each $u\in \X$,
\[
\lim_{h\to 0^+} \|T_h (u)-Du\|_{L^1(\Omega_T)}=\lim_{h\to 0^+} \|\rho_h * Dw-Dw\|_{L^1(\Omega_T)}=0.
\]
Therefore, the spatial gradient operator
$D\colon \mathcal X\to \mathcal  Y$ is the pointwise limit of a sequence of continuous functions $T_h \colon \X\to \Y$; hence $D\colon \mathcal X\to \mathcal  Y$ is  a {\em Baire-one function}.
By Baire's category theorem (e.g., \cite[Theorem 10.13]{BBT}), there exists a
{\em residual set} $\G\subset \mathcal X$ such that the operator $D$ is
continuous at each point of $\mathcal G.$  Since $\X\setminus \G$  is of {\em first category}, the set $\G$ is {\em dense} in $\X$. Therefore, given any $\varphi\in \X,$ for each $\eta>0$, there exists a function $u\in \G$ such that $\|u-\varphi\|_{L^\infty(\Omega_T)}<\eta.$

3. We now prove that each function $u\in \G$ is a weak solution to (\ref{ib-PM}).   Let $u\in \G$ be given. By the density of $\mathcal U_\epsilon$ in $(\X,L^\infty)$, for each $j\in\mathbf N$, there exists a function $ u_j\in\mathcal U_{1/j}$ such that $\|u_j-u\|_{L^\infty(\Omega_T)} <1/j$. Since the operator $D\colon (\X, L^\infty)\to (\Y, L^1)$ is continuous at $u$,  we have  $Du_j\to Du$ in $L^1(\Omega_T;\mathbf R^n).$  Furthermore, from the definition of $\mathcal U_{1/j}$, there exists a vector function $v_j\in W^{1,\infty}_{v^*}(\Omega_T;\mathbf R^n)$ such that, for each $\zeta\in C^\infty(\bar\Omega_T)$ and    $t\in [0,T],$
\begin{equation}\label{div-v3}
\begin{split} \int_\Omega v_j(x,t)\cdot D\zeta(x,t)&dx  =-\int_\Omega u_j(x,t)\zeta(x,t)\,dx,\\
 \|(v_j)_t\|_{L^\infty(\Omega_T)}  \le \frac{1}{2},\quad &\int_0^T\int_\Omega |(v_j)_t-\sigma(Du_j)|\,dxdt \leq\frac{1}{j}|\Omega_T|.\end{split}
\end{equation}
Since $v_j(x,0)=v^*(x,0)\in  W^{1,\infty}(\Omega;\R^n)$, from  $\|(v_j)_t\|_{L^\infty(\Omega_T)}  \le 1/2$, it follows that both sequences $\{v_j\}$ and  $\{(v_j)_t\}$ are bounded in $L^2(\Omega_T;\R^n)\approx L^2((0,T);L^2(\Omega;\R^n)).$ We may  assume $v_j \wcon v $ and $(v_j)_t\wcon v_t$ weakly in $L^2((0,T);L^2(\Omega;\R^n))$ for some $v\in W^{1,2}((0,T);L^2(\Omega;\R^n)).$  Upon taking the limit as $j\to \infty$ in (\ref{div-v3}) and noticing $v\in C([0,T];L^2(\Omega;\R^n)),$ we obtain that
\[
\begin{split}
    \int_\Omega  v(x,t)\cdot &D\zeta(x,t)\,dx     =  -\int_\Omega u(x,t)\zeta(x,t)\,dx \quad (t\in [0,T]), \\
&v_t(x,t)= \sigma(Du(x,t)) \quad   a.e. \; (x,t)\in \Omega_T. \end{split}
\]
Consequently, by Lemma \ref{gen-lem}, $u$ is a weak solution to (\ref{ib-PM}).

4. Finally, assume $\mathcal U$ contains a function that is not a weak solution to (\ref{ib-PM}); hence   $\G\ne \mathcal U.$ Then $\G$ cannot be a finite set since otherwise the $L^\infty$-closure $\X=\overline{\G}  =\overline{\mathcal U}$ would be a  finite set, making  $\mathcal U=\G$; therefore, in this case,   (\ref{ib-PM}) admits infinitely many weak solutions.
This completes the proof.
\end{proof}

The rest of the paper is devoted to the construction of a suitable admissible class $\mathcal U\subset W^{1,\infty}(\Omega_T)$ fulfilling the \emph{density property}:
\begin{equation}\label{density-0}
\mbox{$\mathcal U_\epsilon$ is dense in $\mathcal U$ under the $L^\infty$-norm  for each $\epsilon>0.$}
\end{equation}

\section{Geometric considerations: Relaxation of  $\nabla \omega(z)\in K(0)$}

Let $K(s)$ be the matrix set defined by (\ref{set-K}) above. Since $K(s)$ is a translation of set $K(0),$ we focus on  the set  $K_0=K(0)$; that is,
\begin{equation*}%\label{K0}
K_0=\left\{ \begin{pmatrix} p & c\\ B & \sigma(p)\end{pmatrix}\,\Big|\,  p\in\R^n, \, c\in\R, \, B\in \mathbf M^{n\times n},\, \tr B=0\right\},
\end{equation*}
where $\sigma(p)=\frac{p}{1+|p|^2}$ is the Perona-Malik function.

\subsection{Rank-one lamination of  $K_0$} We first compute certain rank-one structures of the set $K_0.$

Let $L(K_0)$ be the set of all matrices $\xi\in \mathbf M^{(1+n)\times (n+1)}$ that are not in $K_0$ but are representable by $\xi=\lambda\xi_1+(1-\lambda)\xi_2$ for some $\lambda\in (0,1)$ and $\xi_1,\;\xi_2\in K_0$  with  $\rank(\xi_1-\xi_2)=1$, or equivalently,
\[
L(K_0)=\{ \xi\notin K_0\; |\; \mbox{$\xi+t_\pm\eta \in K_0$ for some  $t_-<0<t_+$ and $\rank\eta=1$}\}.
\]
Suppose $\xi=\begin{pmatrix} p & c\\ B & \beta\end{pmatrix}\in L(K_0)$, with $\xi+t_\pm \eta \in K_0,$ where $t_-<0<t_+$ and $\eta$ is a rank-one matrix given by
\[
\eta =\begin{pmatrix}a\\\alpha\end{pmatrix}\otimes (q,b)=\begin{pmatrix}aq & ab\\\alpha\otimes q & b\alpha\end{pmatrix}, \quad a^2+|\alpha|^2\ne 0,\,\, b^2+|q|^2\ne 0,
\]
for some $a,\, b\in \R$ and $\alpha,  \, q\in \R^n$; here $\alpha\otimes q$ denotes  the  rank-one or zero matrix $(\alpha^i q_j)$ in $\N.$

Condition  $\xi+t_\pm \eta \in K_0$ with $t_-<0<t_+$ is equivalent to the following:
\begin{equation}\label{rk-1}
\tr B=0,\quad \alpha\cdot q=0,\quad
\sigma(p+t_\pm aq)=\beta+t_\pm b\alpha.
\end{equation}
If $aq=0$, then $\sigma(p)=\beta +tb\alpha$ has two different solutions of $t$ only when $b\alpha=0,$ but then we would have $\sigma(p)=\beta$ and thus $\xi\in K_0,$ a contradiction.
Therefore, $aq\ne 0.$ By rescaling $\eta$ and $t_\pm$, we assume $a=1$ and $|q|=1;$ namely,
\[
\eta=\begin{pmatrix} q &  b\\\alpha\otimes q & b\alpha\end{pmatrix},\quad |q|=1,\quad \alpha\cdot q=0.
\]

\subsubsection*{Case 1.} Assume $b\alpha=0.$ In this case, by (\ref{rk-1}), the equation $\sigma(p+tq)=\beta$ has two solutions of $t$ of opposite signs and thus we must have $p=xq$ and $\beta=uq$, and $\sigma(xq+tq)=uq$ becomes  a quadratic equation $x+t=u+u (x^2 +2xt+t^2)$, which has two solutions $t=t_\pm$ of opposite signs if and only if  $u\ne 0$ and $x^2 -\frac{x}{u}+1<0;$ this condition can be written as
\begin{equation*}%\label{cond-1}
 |\beta|^2+(p\cdot\beta)^2-p\cdot\beta=u^2+(xu)^2-xu<0.
\end{equation*}

\begin{remk} In this case  one can always select  $\eta=\begin{pmatrix}q & b\\0& 0\end{pmatrix},\; |q|=1,\; b\in \mathbf R.$ This is the case for the one-dimensional problems studied in \cite{KY, Zh, Zh1},  where the existence  results  are primarily proved based on the structure of such $\eta$'s. However, if $n\ge 2$, such $\eta$'s are not sufficient to characterize all the rank-one structures. {\em Case 2} below thus becomes pivotal.
\end{remk}

\subsubsection*{Case 2.} Assume $b\alpha\ne 0;$ so $b\ne 0$ and $\alpha\ne 0.$   In this case, we write
\[
\eta=\begin{pmatrix} q &  b\\\frac{1}{b}\gamma\otimes q & \gamma\end{pmatrix},\quad |q|=1,\; \gamma\cdot q=0,\; \gamma\ne 0,\; b\ne 0.
\]
Since the equation $\sigma(p+tq)=\beta+t\gamma$ has two solutions  $t=t_\pm$ of opposite signs, we must have  $p=xq+y\gamma$ and $\beta=uq +v\gamma,$ and the equation $\sigma(p+tq)=\beta+t\gamma$  becomes a system of two equations:
\begin{equation}\label{sys1}
\begin{cases}
x+t =u(1+(x+t)^2+|\gamma|^2 y^2),\\
 y=(v+t)(1+(x+t)^2+|\gamma|^2 y^2).
\end{cases}
\end{equation}
This system has two solutions $t=t_\pm$ of opposite signs, and thus $u\ne 0$ and $y\ne 0$. So (\ref{sys1}) is equivalent to a system of two {\em quadratic} equations:
\begin{equation}\label{sys2}
\begin{cases}
  t^2+(2x-\frac{1}{u})t+x^2+|\gamma|^2y^2+1-\frac{x}{u}=0,\\
 t^2+(v+x)t+xv-yu=0.
\end{cases}
\end{equation}
The necessary and sufficient condition for (\ref{sys2}) to have   two solutions  $t=t_\pm$ of opposite signs is that the two quadratic equations of $t$ have the same coefficients and the constant terms are negative, which yields that
\[
 x=\frac{1}{u}+v,\quad x^2+|\gamma|^2 y^2+1  -\frac{x}{u}= xv-yu <0.
\]
Here, if $v=0$, then $x=\frac{1}{u}$, and taking this into the inequality, we have $1+|\gamma|^2y^2<0$, a contradiction. So $v\neq 0$.
Therefore
\begin{equation}\label{cond-20}
 {uv}={xu-1},\quad
|\gamma|^2 = \frac{1-xu}{yv} -\frac{1}{y^2},
\end{equation}
and
\begin{equation}\label{cond-21}
 xv-yu=\left( \frac{x}{u}-\frac{y}{v}\right)(xu-1)<0.
\end{equation}

We now solve for $x,\,y,\,u,\,v$ from (\ref{cond-20}) in terms of $p$ and $\beta.$ From $p=xq+y\gamma,\; \beta=uq +v\gamma$, it follows that
\begin{equation}\label{eq-q-g}
q=\frac{1}{xv-yu}(vp-y\beta),\quad \gamma =\frac{1}{xv-yu}(-up+x \beta).
\end{equation}
By (\ref{cond-20}) and (\ref{eq-q-g}), we have
\[
p\cdot \beta= xu+yv |\gamma|^2=1-\frac{v}{y}, \quad
\frac{x}{u}= \frac{\frac{v}{y}|p|^2-p\cdot \beta}{\frac{v}{y} p\cdot \beta - |\beta|^2},
\]
where $\frac{v}{y} p\cdot \beta - |\beta|^2\neq 0$ by (\ref{cond-21}).
Let $k=x/u,\;l=y/v.$ Then
\begin{equation}\label{def-other1}
l=\frac{1}{1-p\cdot \beta},\quad
k =\frac{(1-p\cdot\beta)|p|^2-p\cdot \beta}{(1-p\cdot\beta)p\cdot \beta - |\beta|^2}.
\end{equation}
Moreover,
\[
k-l=\frac{x}{u}-\frac{y}{v}=\frac{|p|^2-l p\cdot \beta}{p\cdot \beta -l|\beta|^2}-l
=\frac{|p-l\beta|^2}{p\cdot \beta -l|\beta|^2}.
\]
From $|q|=1$, we have $xv-yu=(k-l)uv=-|vp-lv\beta|$ and hence
\begin{equation}\label{def-u}
u=-\mbox{sgn}(v) \frac{|p-l\beta|}{k-l}=-\mbox{sgn}(v) \frac{p\cdot \beta -l |\beta|^2}{|p-l\beta|},
\end{equation}
where $\mbox{sgn}(v)=v/|v|$ is the sign of $v\ne 0.$
We then obtain $x,\,v,\,y$ by
\begin{equation}\label{def-other2}
x=ku,\quad v=x-\frac{1}{u}=ku-\frac{1}{u},\quad y=lv=lku-\frac{l}{u}.
\end{equation}
In this way, we have solved  $x,\,y,\,u,\,v$ in terms of $p,\,\beta$, uniquely up to the sign change. We can check that both conditions in (\ref{cond-20}) are satisfied.

Let us consider inequality (\ref{cond-21}) for  these solutions.
Equation on $|\gamma|^2$ in (\ref{cond-20}) implies  $\frac{y}{v}(1-xu)>1$ and hence the inequality  (\ref{cond-21}) yields  $\frac{x}{u}(1-xu)>\frac{y}{v}(1-xu)>1.$ So $0<xu<1$ and thus $x/u>y/v>1,$ i.e.,  $k>l>1.$ Then we deduce  the inequality
\begin{equation}\label{sys4}
|\beta|^2 + (p\cdot \beta)^2 -p\cdot \beta <0.
\end{equation}

\subsection{Exact formula of $L(K_0)$} In fact,  inequality (\ref{sys4})  exactly  characterizes the set $L(K_0).$
We have the following result.

\begin{thm}\label{thm-rconv}
\begin{equation}\label{set-L}
L(K_0)=\left\{ \begin{pmatrix} p & c\\ B & \beta\end{pmatrix} \, \Big |\; \tr B=0, \; |\beta|^2 +(p\cdot\beta)^2 -p\cdot \beta <0\right \}.
\end{equation}
Moreover, given any $\xi\in L(K_0)$, there exist
a rank-one matrix \[
\eta=\begin{pmatrix}
                            q & b \\
                            \frac{1}{b}\gamma\otimes q & \gamma \\
                         \end{pmatrix}
\]
with $|q|=1,\;  \gamma\cdot q=0,\;  b\neq0$ and two numbers $t_-<0<t_+$ such that
\[
\xi+t_\pm\eta\in K_0,%\quad 0<|c-c_\pm|<\epsilon(t_+-t_-).
\]
where $|b|>0$ can be arbitrarily small.
\end{thm}

\begin{proof}  Let $S$ be the set defined on the right-hand side of (\ref{set-L}).  The previous calculations show that $L(K_0)\subseteq S.$ To verify the reverse inclusion $S\subseteq L(K_0),$  let $\xi=\begin{pmatrix} p & c\\ B & \beta\end{pmatrix}\in S$.
Then $|\beta|^2+(p\cdot\beta)^2 -p\cdot \beta < 0.$
So $0<p\cdot\beta<1$ and $(1-p\cdot\beta)p\cdot\beta-|\beta|^2>0,$ and hence we can define $l,\;k$ by (\ref{def-other1}), so that $l>0$, $k>0$. To fix the sign, we define $u$ by (\ref{def-u}) with $+$ sign:
\[
u=\frac{p\cdot\beta-l|\beta|^2}{|p-l\beta|}=\frac{(1-p\cdot\beta)p\cdot\beta-|\beta|^2}{|(1-p\cdot\beta)p-\beta|}>0.
\]
We now define $x, v, y$ by (\ref{def-other2}). Then $x>0,$ $v<0$, and $y<0$.
After deducing that $xv-yu<0$, we finally define $q,\gamma$ by (\ref{eq-q-g}).
It is then straightforward to check the following:
\[
p=xq+y\gamma,\quad \beta=uq+v\gamma,\quad |q|=1,\quad \gamma\cdot q=0,
\]
\[
\quad x=\frac{1}{u}+v,\quad x^2+|\gamma|^2y^2+1-\frac{x}{u}=xv-yu<0.
\]
In particular, equation $\sigma(p+tq)=\beta+t\gamma$ has two solutions $t=t_\pm$ with $t_-<0<t_+$.
Now  let $
\eta=\begin{pmatrix} q &  b\\\frac{1}{b}\gamma\otimes q & \gamma\end{pmatrix}$, where $b\ne 0$ is arbitrary. Then $\xi+t_\pm \eta\in K_0,$ and so $\xi\in L(K_0)$.
The proof is now complete.
\end{proof}

\begin{remk} The quantities  $q\in\mathbf{S}^{n-1}$, $\gamma\in\R^n$  and $t_\pm$ defined in the proof  depend  \emph{continuously}  on $(p,\beta)\in\R^{n+n}$ with $|\beta|^2 +(p\cdot\beta)^2 -p\cdot \beta <0.$  We may also take $b>0$ (or $b<0$) to be a continuous function of all such $(p,\beta)$.
\end{remk}

\subsection{The approximating sets $\mathcal S_\delta$ and $S_\delta$}  Given any  $0\leq\delta\leq 1/2$, let
\begin{equation}\label{S-delta}
 \mathcal{S}_\delta=\{(p,\beta)\in\R^{n+n}\;|\;\delta |(1-p\cdot\beta)p-\beta|+|\beta|^2 +(p\cdot\beta)^2 -p\cdot \beta < 0\},
\end{equation}
\[
S_\delta=\left\{ \begin{pmatrix} p & c\\ B & \beta\end{pmatrix} \, \Big |\; \tr B=0, \; (p,\beta)\in\mathcal S_\delta
\right\}.
\]
Then $S_0=L(K_0)$. Immediate properties of the open
sets $\mathcal{S}_\delta$ are that
\[
\mathcal{S}_{\delta_2}\subset \mathcal{S}_{\delta_1}\quad\textrm{for $0\leq\delta_1<\delta_2\leq 1/2$},\quad \mathcal{S}_{1/2}=\emptyset,
\]
\[
\mathcal{S}_\delta\neq\emptyset\quad\textrm{for $0\leq\delta< 1/2$},\quad \mathcal{S}_0=\bigcup_{0<\delta<1/2}\mathcal{S}_\delta.
\]
In what follows we always assume $0<\delta<1/2$ unless otherwise stated.

 \begin{pro}\label{pro:rk-1-line} Let $\xi\in S_{\delta}$ and   $\xi_\pm \in K_0$ with $\rank(\xi_+-\xi_-)=1$ satisfy  that $\xi$ lies in the open line segment $(\xi_-,\xi_+)$. Then $(\xi_-,\xi_+)\subset S_\delta$.
\end{pro}
\begin{proof} Consider functions
\[
F(\xi)= |(1-p\cdot\beta)p-\beta|,\;\; G(\xi)=|\beta|^2 +(p\cdot\beta)^2 -p\cdot \beta, \quad \forall\;  \xi=\begin{pmatrix} p & c\\ B & \beta\end{pmatrix}.
\]
 Then both $F$ and $G$ vanish on set $K_0$. For the  given $\xi$ and $\xi_\pm$, let $f(\tau)=F(\xi^\tau)$ and $g(\tau)=G(\xi^\tau),$ where  $\xi^\tau=\tau\xi_++(1-\tau)\xi_-.$ The rank-one condition implies that the corresponding term $p^\tau\cdot \beta^\tau$ is linear in $\tau$; hence $g(\tau)$ is a quadratic polynomial of $\tau$ and  $f(\tau)$ is the length of a vector  quadratic in $\tau.$ Since both $f(\tau)$ and $g(\tau)$ vanish when $\tau=0$ and 1, we must have
$g(\tau)=C_1\tau(1-\tau)$ and $f(\tau)=C_2\tau(1-\tau)$ for some constants $C_1, C_2$. Since $\xi=\xi^\lambda\in S_\delta$ for some $0<\lambda<1$, we have
\[
\delta f(\lambda)+g(\lambda)=\lambda(1-\lambda)(\delta C_2 +C_1)<0.
\]
This  implies the constant $\delta C_2 +C_1$ is negative.  Hence $\delta f(\tau)+g(\tau)=\tau(1-\tau)(\delta C_2 +C_1)<0$ for all $\tau\in (0,1),$ which proves exactly  $\xi^\tau\in S_\delta$ for all $\tau\in (0,1).$
\end{proof}

\begin{remk} The {\em lamination convex hull} $K^{lc}$ of a set $K\subset\mathbf{M} ^{m\times n}$ is defined to be the smallest  set $S\subset\mathbf{M} ^{m\times n}$ containing $K$ with the property that if $\xi_\pm\in S$ with $\rank(\xi_+-\xi_-)=1$ then $(\xi_-,\xi_+)\subset S$ (see \cite{Da}). As  in the proof of  Proposition \ref{pro:rk-1-line}, one can see that $K_0^{lc}=K_0\cup L(K_0).$
\end{remk}

\begin{lem}\label{bound-lem}  Let $m_\pm(\delta)$ be defined by $(\ref{root-PM})$ above. It follows that
\[
\sup_{(p,\beta)\in\mathcal S_\delta}|p|=m_+(\delta),\quad \inf_{(p,\beta)\in\mathcal S_\delta}|p|=m_-(\delta).
\]
\end{lem}
\begin{proof}
Let $(p,\beta)\in\mathcal S_\delta$. Then   $0<p\cdot\beta<1$ and $\beta\cdot((1-p\cdot\beta)p-\beta)>0.$ We write
\[
p\cdot\beta=|p| |\beta|\cos\theta,\;\; \beta\cdot((1-p\cdot\beta)p-\beta)=|\beta||(1-p\cdot\beta)p-\beta|\cos\theta'
\]
for some $0\leq\theta, \; \theta'<\pi/2$. A simple geometry shows that $\theta\le \theta'.$ As $|\beta|^2 +(p\cdot\beta)^2 -p\cdot \beta < 0$, we have
\[
|\beta|^2+|p|^2|\beta|^2\cos^2\theta-|p||\beta|\cos\theta<0,
\]
and so,
\[
|\beta|<\frac{1}{2},\quad \frac{1-\sqrt{1-4|\beta|^2}}{2|\beta|\cos\theta}
<|p|<\frac{1+\sqrt{1-4|\beta|^2}}{2|\beta|\cos\theta}.
\]
Condition $(p,\beta)\in\mathcal S_\delta$ becomes  $\delta<|\beta|\cos\theta';$ so $|\beta|\ge|\beta|\cos\theta >\delta$ and
\[
\begin{split} \frac{1-\sqrt{1-4\delta^2}}{2\delta}<\frac{1-\sqrt{1-4|\beta|^2}}{2|\beta|}\leq\frac{1-\sqrt{1-4|\beta|^2}}{2|\beta|\cos\theta} \\
 <|p|<\frac{1+\sqrt{1-4|\beta|^2}}{2|\beta|\cos\theta}<\frac{1+\sqrt{1-4\delta^2}}{2\delta}. 
\end{split}
\]
This proves
\begin{equation}\label{eq-ll}
\frac{1-\sqrt{1-4\delta^2}}{2\delta}\leq \inf_{(p,\beta)\in\mathcal S_\delta}|p| \leq \sup_{(p,\beta)\in\mathcal S_\delta}|p| \leq \frac{1+\sqrt{1-4\delta^2}}{2\delta}.
\end{equation}
Next, fix any $\beta\in\R^n$ with $\delta<|\beta|<1/2$. Let $l, l'$ be any numbers satisfying
\begin{equation}\label{eq-l}
\frac{1-\sqrt{1-4|\beta|^2}}{2|\beta|}<l<l'<\frac{1+\sqrt{1-4|\beta|^2}}{2|\beta|},
\end{equation}
and $p=\frac{l}{|\beta|}\beta$, $p'=\frac{l'}{|\beta|}\beta$. Then   $(p,\beta)$, $(p',\beta)$ are both in $\mathcal S_\delta$ with $|p|=l,   |p'|=l'$; so
\[
\inf_{(p,\beta)\in\mathcal S_\delta}|p| \le l<l'\le \sup_{(p,\beta)\in\mathcal S_\delta}|p|.
\]
As $l,l'$ are arbitrary and satisfy (\ref{eq-l}), we have
\[
\inf_{(p,\beta)\in\mathcal S_\delta}|p|  \leq\frac{1-\sqrt{1-4|\beta|^2}}{2|\beta|}<\frac{1+\sqrt{1-4|\beta|^2}}{2|\beta|}\leq \sup_{(p,\beta)\in\mathcal S_\delta}|p|.
\]
Finally, taking $|\beta|\to\delta^+$ and combining with (\ref{eq-ll}) complete  the proof.
\end{proof}

As an immediate consequence of the previous lemma, we have

\begin{coro}\label{bound-coro}  
$\mathcal{S}_\delta \subset  \{(p,\beta) \;|\; m_-(\delta)<|p|<m_+(\delta),\; \delta<|\beta|<1/2\}.$
\end{coro}

\subsection{A useful convex integration  lemma}
The following result is important for convex integration with linear constraint.    For a more general case, see \cite[Lemma 2.1]{Po}.

\begin{lem}\label{approx-lem} Let $\lambda_1,\lambda_2>0$ and $\eta_1=-\lambda_1\eta, \; \eta_2=\lambda_2\eta$ with
\[
\eta=\begin{pmatrix} q &  b\\\frac{1}{b}\gamma\otimes q & \gamma\end{pmatrix},\quad |q|=1,\; \gamma\cdot q=0,\;   b\ne 0.
\]
Let $G\subset \mathbf R^{n+1}$ be a bounded domain. Then for each $\epsilon>0$, there exists a function $\omega=(\varphi,\psi)\in C_c^{\infty}(\R^{n+1};\mathbf R^{1+n})$ with $\mathrm{supp}(\omega)\subset\subset G$  that satisfies the following properties:
\begin{enumerate}
\item[(a)] \; $\dv \psi =0$  in $G$,

\item[(b)] \; $|\{z\in G\;|\; \nabla\omega(z)\notin \{\eta_1,\;\eta_2\}\}|<\epsilon,$

\item[(c)] \; $\dist(\nabla \omega(z),[\eta_1,\eta_2])<\epsilon$ for all $z\in G,$

\item[(d)] \; $\|\omega\|_{L^\infty(G)}<\epsilon,$

\item[(e)] \; $\int_{\R^n}\varphi(x,t)\,dx=0$ for each $t\in\R$.
\end{enumerate}

\end{lem}

\begin{proof} 
1. The proof follows a simplified version of \cite[Lemma 2.1]{Po}. Define a map $\mathcal P\colon C^1(\mathbf R^{n+1})\to C^0(\mathbf R^{n+1};\mathbf R^{1+n})$ by setting $\mathcal P(h)=(u,v)$, where, for $h(x,t)\in C^1(\mathbf R^{n+1})$,
\[
u(x,t)= q\cdot Dh(x,t), \quad
 v(x,t)=\frac{1}{b}(\gamma\otimes q-q\otimes \gamma)Dh(x,t).
\]
We easily see  that $\mathcal P(h)=(u,v)\in C_c^\infty(\mathbf R^{n+1};\mathbf R^{1+n}),$ $\mathrm{supp}(\mathcal{P}(h))\subset\mathrm{supp}(h),$ $\dv v\equiv 0$,  and $\int_{\R^n}u(x,t)\,dx=0$ for all $t\in\R$,  for all $h\in C_c^\infty\mathbf (\R^{n+1}).$ For $h(x,t)=f(q\cdot x+bt)$ with $f\in C^\infty(\R)$, $w=(u,v)=\mathcal P(h)$ is given by $u(x,t)=f'(q\cdot x+bt)$ and $v(x,t)=f'(q\cdot x+bt)\frac{\gamma}{b}$, and hence $\nabla w(x,t)=f''(q\cdot x+bt)\eta.$
We also note that $\mathcal P(gh)=g\mathcal P(h)+h\mathcal P(g)$ and hence
\begin{equation}\label{bi-form}
\nabla \mathcal P(gh)= g\nabla \mathcal P(h) + h\nabla \mathcal P(g) + \mathcal B(\nabla g,\nabla h) \quad \forall\; g,\;h\in C^\infty(\R^{n+1}),
\end{equation}
where $\mathcal B(\nabla g,\nabla h)$ is a bilinear map of $\nabla g$ and $\nabla h$; so $|\mathcal B(\nabla h,\nabla g)|\le C|\nabla h||\nabla g|$ for some constant $C>0.$

2. Let $G_\epsilon\subset\subset G$ be a smooth sub-domain such that $|G\setminus G_\epsilon|<\epsilon/2,$ and let $\rho_\epsilon\in C^\infty_c(G)$ be a cut-off function satisfying $0\le \rho_\epsilon\le 1$ in $G$, $\rho_\epsilon =1$ on $G_\epsilon$.
As $G$ is bounded, $G \subset\{(x,t)\;|\; k<q\cdot x+bt<l\}$ for some numbers $k<l.$ For each $\tau>0$, we can find a function $f_\tau\in C^\infty_c(k,l)$ satisfying
\[
-\lambda_1\le f_\tau'' \le \lambda_2, \;  |\{s\in (k,l)\;|\;f_\tau''(s)\notin\{-\lambda_1,\;  \lambda_2\}\}|<\tau, \; \|f_\tau\|_{L^\infty}+ \|f_\tau'\|_{L^\infty} <\tau.
\]

3. Define $\omega=(\varphi,\psi)=\mathcal P(\rho_\epsilon(x,t) h_\tau(x,t)),$ where $h_\tau(x,t)= f_\tau(q\cdot x+ bt).$ Then  $
\|h_\tau\|_{C^1}\le C\|f_\tau\|_{C^1}\le C\tau$,  $\omega\in C_c^{\infty}(\R^{n+1};\mathbf R^{1+n})$,  $\mathrm{supp}(\omega)\subset\mathrm{supp}(\rho_\epsilon)\subset\subset G$, and (a) and (e) are satisfied. Note that
\[
|\omega|\le |\rho_\epsilon||\mathcal P(h_\tau)|+|h_\tau||\mathcal P(\rho_\epsilon)|\le C_\epsilon \tau,
\]
where $C_\epsilon>0$ is a constant depending on $\|\rho_\epsilon\|_{C^1(G)}$. So we can choose a $\tau_1>0$ so small that (d) is satisfied for all $0<\tau<\tau_1.$ Note also that
\[
\{z\in G\,|\; \nabla\omega(z)\notin \{\eta_1,\eta_2\}\}\subseteq (G\setminus G_\epsilon)\cup  \{z\in G_\epsilon \,|\; f''_\tau(q\cdot x+bt)\notin \{-\lambda_1,\; \lambda_2\}\}.
\]
Since $| \{z\in G_\epsilon \,|\; f''_\tau(q\cdot x+bt)\notin \{-\lambda_1,\; \lambda_2\}| \le N |\{s\in (k,l)\;|\;f_\tau''(s)\notin\{-\lambda_1,\;  \lambda_2\}\}|$ for some constant $N>0$ depending only on set $G$, there exists a $\tau_2>0$ such that
\[
|\{z\in G\,|\; \nabla\omega(z)\notin \{\eta_1,\eta_2\}\}|\le \frac{\epsilon}{2}+ N \tau<\epsilon
\]
for all $0<\tau<\tau_2$. Therefore, (b) is satisfied. Finally, note that
\[
\rho_\epsilon \nabla \mathcal P(h_\tau(x,t))=\rho_\epsilon f''_\tau(q\cdot x+bt) \eta\in [\eta_1,\eta_2] \quad \mbox{in $G$}
\]
and, by (\ref{bi-form}), for all $z=(x,t)\in G$,
\[
|\nabla \omega(z) -\rho_\epsilon \nabla \mathcal P(h_\tau(x,t))|\le |h_\tau| |\nabla \mathcal P(\rho_\epsilon)|+|\mathcal B(\nabla h_\tau,\nabla \rho_\epsilon)|\le C_\epsilon' \tau<\epsilon
\]
for all $0<\tau<\tau_3$, where $C'_\epsilon>0$ is a constant depending on  $\|\rho_\epsilon\|_{C^2(G)}$, and $\tau_3>0$ is another constant. Hence (c) is satisfied. Taking $0<\tau<\min\{\tau_1,\tau_2,\tau_3\}$, the proof is complete.
\end{proof}

\subsection{Relaxation of   $\nabla \omega(z)\in K_0$}
We now study  the relaxation of  homogeneous   inclusion$\nabla \omega(z)\in K_0.$ We prove the following result  in a form  suitable for   later use.

\begin{thm}\label{main-lemma}   Let $\mathcal K$ be a compact subset of $\mathcal S_\delta.$ Let $\tilde Q\times \tilde I$ be a  box  in $\mathbf R^{n+1}.$  Then, given any $\epsilon>0$, there exists a $\rho_0>0$ such that for each box $Q\times I\subset \tilde Q\times \tilde I$, point  $(p,\beta)\in\mathcal K$,  and  number $\rho>0$ sufficiently small, there exists a function $\omega=(\varphi,\psi)\in C^\infty_{c} (Q\times I;\mathbf R^{1+n})$  satisfying the following properties:
\begin{enumerate}
\item[(a)] \; $\dv \psi =0$  in $Q\times I$,

\item[(b)] \; $(p'+D\varphi(z), \beta'+ \psi_t(z))\in \mathcal{S}_{\delta}$ for all $z \in Q\times I$ and $|(p',\beta')-(p,\beta)|\leq \rho_0,$

\item[(c)] \; $\|\omega\|_{L^\infty(Q\times I)}<\rho,$

\item[(d)] \; $\int_{Q\times I} |\beta+\psi_t(z)-\sigma(p+D\varphi(z))|dz<\epsilon {|Q\times I|}/{|\tilde Q\times \tilde I|},$

\item[(e)] \; $\int_{Q} \varphi(x,t)dx=0$  for all $t\in I,$

\item[(f)] \; $\|\varphi_t\|_{L^\infty(Q\times I)}<\rho.$
\end{enumerate}
\end{thm}

\begin{proof} Let $\xi=\xi(p,\beta)=\begin{pmatrix} p & 0\\O & \beta\end{pmatrix}$ for $(p,\beta)\in\mathcal K\subset\mathcal S_\delta$, where $O$ is  the $n\times  n$ zero matrix. We omit the dependence on $(p,\beta)$ in the following whenever it is clear from the context.  Since $\xi\in S_{\delta}\subset L(K_0)$ on $\mathcal K$, it follows from Theorem \ref{thm-rconv} and its remark that given any $\rho>0$, there exist
\emph{continuous} functions $q:\mathcal K\to\mathbf{S}^{n-1},$ $\gamma:\mathcal K\to\R^n$,
$t_\pm:\mathcal K\to \R$, and $b:\mathcal K\to (0,\infty)$ with $\gamma\cdot q=0$, $t_-<0<t_+$ on $\mathcal K$ such that letting  $\eta =\begin{pmatrix} q & b \\ \frac{1}{b}\gamma\otimes q & \gamma\end{pmatrix}  $ on $\mathcal K$, we have
\[
\xi+t_\pm\eta\in K_0,\quad 0<b<\frac{\rho}{t_+-t_-}\quad\textrm{on $\mathcal K$}.
\]
Writing $\xi_\pm=\begin{pmatrix} p_\pm & c_\pm\\B_\pm & \beta_\pm\end{pmatrix}=\xi+t_\pm\eta$ on $\mathcal K$, we have $\xi=\lambda \xi_+ + (1-\lambda)\xi_-,\;\lambda=\frac{-t_-}{t_+-t_-}\in(0,1)$ on $\mathcal K$.

Proposition \ref{pro:rk-1-line} implies that  on $\mathcal K$,  both $\xi_-^\tau=\tau \xi_+ + (1-\tau)\xi_-$ and $\xi_+^\tau=(1-\tau) \xi_+ +  \tau \xi_-$ belong to $S_{\delta}$ for each $\tau\in (0,1)$.  Let $0<\tau<\min_{\mathcal K}\min\{\lambda,1-\lambda\}\le \frac12$ be a small number to be selected later.   Let $\lambda'=\frac{\lambda-\tau}{1-2\tau}$ on $\mathcal K$. Then $\lambda'\in (0,1)$,
$\xi=\lambda'\xi_+^\tau+(1-\lambda')\xi_-^\tau$ on $\mathcal K$.
Moreover, on $\mathcal K$, $\xi_+^\tau-\xi_-^\tau=(1-2\tau)(\xi_+ -\xi_- )$ is rank-one, $[\xi_-^\tau,\xi_+^\tau]\subset(\xi_-,\xi_+)\subset S_{\delta}$, and
\[
c\tau\leq|\xi_+^\tau-\xi_+|=|\xi_-^\tau-\xi_-|=\tau |\xi_+ -\xi_-|=\tau(t_+-t_-)|\eta|\leq C\tau,
\]
where $C=\max_{\mathcal K}(t_+-t_-)|\eta|\geq\min_{\mathcal K}(t_+-t_-)|\eta|=c>0.$ By continuity, $H_\tau=\cup_{(p,\beta)\in\mathcal{K}}[\xi_-^\tau(p,\beta),\xi_+^\tau(p,\beta)]$ is a compact subset of $S_\delta$, where $S_\delta$ is open in
\[
\Sigma_0=\left\{\begin{pmatrix} p & c\\B & \beta\end{pmatrix}\;\Big|\; \mathrm{tr}B=0\right\}.
\]
So $d_\tau=\mathrm{dist}(H_\tau,\partial|_{\Sigma_0}S_\delta)>0$, where $\partial|_{\Sigma_0}$ is the relative boundary in the space $\Sigma_0$.

Let $\eta_1=-\lambda_1\eta=-\lambda'(1-2\tau)(t_+-t_-)\eta,\,\eta_2= \lambda_2\eta=(1-\lambda')(1-2\tau)(t_+-t_-)\eta$ on $\mathcal K$, where $\lambda_1=\tau(-t_+)+(1-\tau)(-t_-)>0,\,\lambda_2=(1-\tau)t_++\tau t_->0$  on $\mathcal K$ with $\tau>0$ sufficiently small. Applying  Lemma \ref{approx-lem} to matrices $\eta_1,\,\eta_2$ and set $G= Q\times I$, we obtain that for each $\rho>0$, there exists
a function $\omega=(\varphi,\psi)\in C^\infty_c( Q\times I;\mathbf R^{1+n})$  and an open set $G_\rho\subset\subset  Q\times I$ satisfying the following conditions:
\begin{equation}\label{approx-1}
\begin{cases} \mbox{(1) \;  $\dv \psi =0$  in $Q\times I$,}\\
\mbox{(2) \; $|(Q\times I) \setminus G_\rho|<\rho$;\; $\xi+\nabla  \omega(z)\in \{\xi^\tau_-,\;\xi_+^\tau\}$ for all $z\in  G_\rho$,}\\
\mbox{(3) \; $\xi+\nabla  \omega(z)\in [\xi_-^\tau,\xi_+^\tau]_\rho$ for all $z\in Q\times I,$}\\
\mbox{(4) \; $\|\omega\|_{L^\infty(Q\times I)}<\rho$,} \\
\mbox{(5) \; $\int_Q \varphi(x,t)\,dx=0$ for all $t\in I$,} \\
\mbox{(6) \;  $\|\varphi_t\|_{L^\infty(Q\times I)}<2\rho,$}
\end{cases}
\end{equation}
where  $[\xi_-^\tau,\xi_+^\tau]_\rho$ denotes the $\rho$-neighborhood of closed line segment $[\xi_-^\tau,\xi_+^\tau].$
From (\ref{approx-1}.3), (\ref{approx-1}.6) follows as
\[
|\varphi_t|<|c_+-c_-|+\rho=(t_+-t_-)|b|+\rho<2\rho \quad\textrm{in $Q\times I$.}
\]
By (\ref{approx-1}.3), $|\beta+\psi_t(z)|\le C+\rho$  for $z\in Q\times I;$
hence
\begin{equation}
\begin{split}
&\int_{ Q\times I}   |\beta+\psi_t -\sigma(p+D\varphi)|dz \\
 &\le \int_{G_\rho}|\beta+\psi_t -\sigma(p+D\varphi)|dz + (C+\rho+\frac{1}{2})\rho
\\
&\le |Q\times I| \max\{|\beta_\pm^\tau -\sigma(p_\pm^\tau)|\} +(C+\rho+\frac{1}{2})\rho \\
%&\le |Q\times I| \max\{|\beta_\pm^\tau -\beta_\pm|+|\sigma(p_\pm) -\sigma(p_\pm^\tau)|\} +(C+\rho+\frac{1}{2})\rho \\
& \le C|Q\times I|\tau+ |Q\times I| \max\{|\sigma(p_\pm) -\sigma(p_\pm^\tau)|\}+(C+\rho+\frac{1}{2})\rho,\label{approx-2}
\end{split}
\end{equation}
where  $\xi^\tau_\pm=\begin{pmatrix} p^\tau_\pm & c^\tau_\pm\\B^\tau_\pm & \beta^\tau_\pm\end{pmatrix}.$

Note (a), (c), (e), and (f) follow from (\ref{approx-1}), where $2\rho$ in  (\ref{approx-1}.6) can be adjusted to $\rho$ as in (f). By the uniform continuity of $\sigma$ on the set $J=\{p'\in\R^n\,|\,|p'|\le m_+(\delta)+1\}$, we can find a $\rho'>0$ such that $|\sigma(p')-\sigma(p'')|<\frac{\epsilon}{3|\tilde Q\times\tilde I|}$ whenever $p',\,p''\in J$ and $|p'-p''|<\rho'$, where $m_+(\delta)>0$ is the number defined in Lemma \ref{bound-lem}.  We then choose a $\tau>0$ so small that $C\tau<\rho'$ and $C|\tilde Q\times\tilde I|\tau<\frac{\epsilon}{3}$. Since $p_\pm,\,p_\pm^\tau\in J$ and $|p_\pm-p_\pm^\tau|\le C\tau<\rho'$, it follows from (\ref{approx-2}) that (d) holds for any choice of $\rho>0$ with $(C+\rho+\frac{1}{2})\rho<\frac{\epsilon|Q\times I|}{3|\tilde Q\times\tilde I|}$.
Next, we choose a $\rho_0>0$ such that
$
\rho_0<\frac{d_\tau}{2}.
$
If $0<\rho<\rho_0$, then by (\ref{approx-1}.1) and (\ref{approx-1}.3),
for all $z\in Q\times I$ and $|(p',\beta')-(p,\beta)|\le \rho_0$,
\[
\xi(p',\beta')+\nabla\omega(z)\in\Sigma_0,\quad\mathrm{dist}(\xi(p',\beta')+\nabla\omega(z),H_\tau)<
d_\tau,
\]
and so $\xi(p',\beta')+\nabla\omega(z)\in S_\delta,$ that is, $(p'+D\varphi(z), \beta'+ \psi_t(z))\in \mathcal{S}_{\delta}$. Thus (b) holds.

The proof is now complete.
\end{proof}

\section{Construction of admissible class $\mathcal U$}

In this section, we define a suitable admissible class $\mathcal U$ as required in Section 3. Assume $\Omega$ and $u_0$ are as given  in Theorem \ref{thm:main-1}, with (\ref{ib-PM-2}) fulfilled in addition.  
 Let $M=\|Du_0\|_{L^\infty(\Omega)}>0$.

\subsection{The modified uniformly parabolic problem} We first apply Lemma \ref{lem-modi} to  construct the function $f\in C^{3}([0,\infty))$ with a suitable choice of $\delta\in (0,1/2)$ and $1<\Lambda<m_+(\delta)$ according to the value of $M$ as follows (see Figure 1).

\begin{enumerate}
\item[(i)]  If $0<M<1$,  we select $0<\delta<\frac{M}{1+M^2}$ and arbitrary  $1<\Lambda<m_+(\delta).$ 
\item[(ii)] If $M\ge 1$ and $\lambda >0$,   we  select $\delta=\frac{M+\lambda}{1+(M+\lambda)^2}$ and   arbitrary $\Lambda \in  (M, M+\lambda)$;  in this case,  $m_+(\delta)=M+\lambda.$  
\end{enumerate}
Note that in both cases we have $M<\Lambda.$
Once   $f\in C^{3}([0,\infty))$ is constructed, we define
\[
A(p)=f(|p|^2)p \quad (p\in \mathbf R^n).
\]
By Lemma \ref{lem-modi}, equation $u_t=\dv A(Du)$ is uniformly parabolic.  We have  the following result.

\begin{lem}\label{lem-a} With $\delta$ selected above and $\mathcal S_\delta$ defined by $(\ref{S-delta})$,  one has
\[
(p,A(p))\in \mathcal S_\delta \quad \forall\; m_-(\delta)<|p|\leq M.
\]
\end{lem}
\begin{proof} From the definition of set $\mathcal S_\delta$, it follows   that $(p,A(p))\in \mathcal S_\delta$ if and only if
\[
\frac{\delta}{|p|} < f(|p|^2)<\frac{1}{1+|p|^2};  \;\textrm{namely,} \; \rho(m_-(\delta))=\delta<\rho^*(|p|)<\rho(|p|).
\]
By Lemma \ref{lem-modi}, this condition is satisfied if $m_-(\delta) <|p|\le \Lambda.$ 
\end{proof}

\subsection{The suitable boundary function $\Phi$} By Theorem \ref{existence-gr-max}, the  initial-Neumann boundary value problem
\begin{equation}\label{ib-para}
\begin{cases} u^*_t =\dv (A(Du^*))&\mbox{in $\Omega_T$,}\\
\partial u^*/\partial \n  =0 & \mbox{on $\partial \Omega\times [0,T],$}  \\
u^*(x,0)=u_0(x), &x\in \Omega
\end{cases}
\end{equation}
admits a unique classical solution  $u^*\in C^{2+\alpha,\frac{2+\alpha}{2}}(\bar\Omega_T)$  satisfying
\[
|D u^*(x,t)|\le M,\quad (x,t)\in\Omega_T.
\]

From conditions (\ref{ib-PM-1}) and (\ref{ib-PM-2}),  we can find a function $h\in C^{2+\alpha}(\bar\Omega)$ satisfying  
\[
 \Delta h=u_0 \;\; \mbox{in $\Omega$},\quad  \partial h/\partial \n  =0 \;\; \mbox{on $\partial\Omega.$} 
\]   
Now let $v_0=Dh\in C^{1+\alpha}(\bar\Omega;\R^n)$ and define,  for $(x,t)\in\Omega_T$,
\begin{equation}\label{def-v}
v^*(x,t)=v_0(x)+\int_0^t A(Du^*(x,s))\,ds.
\end{equation}
Define $\Phi = (u^*,v^*)\in C^{1}(\bar\Omega_T;\R^{1+n}).$ Then it is easy to see that $\Phi$ satisfies condition (\ref{bdry-1}) above; i.e.,
\begin{equation*}%\label{bdry-2}
\begin{cases} u^*(x,0)=u_0(x) \; (x\in\Omega),\\
\dv v^*=u^*\;\;\textrm{in $\Omega_T$}, \\
v^*(\cdot,t) \cdot \mathbf n|_{\partial\Omega} =0 \; \; \forall\; t\in [0,T].\end{cases}
\end{equation*}

\begin{lem} \label{lem-b} Let 
\[
\mathcal K_\delta= \{(p,\sigma(p))\;|\; |p|\le m_-(\delta) \}.
\]
Then
\[
(Du^*(x,t), v^*_t(x,t))\in \mathcal S_\delta \cup \mathcal K_\delta\quad \forall\; (x,t)\in \Omega_T.
\]
\end{lem}
\begin{proof}
Given $(x,t)\in\Omega_T$, let $p=Du^*(x,t)$; then  $|p|\le M.$  By (\ref{def-v}), $v^*_t(x,t)=A(p).$ If $|p|\le m_-(\delta)$, then $A(p)=\sigma(p)$ and hence
\[
(Du^*(x,t),v_t^*(x,t))=(p,A(p))=(p,\sigma(p))\in \mathcal K_\delta.
\]
If $m_-(\delta)<|p|\le M,$  then by Lemma \ref{lem-a}
\[
(Du^*(x,t),v_t^*(x,t))=(p,A(p))\in \mathcal S_\delta.
\]
Hence $(Du^*,v_t^*)\in\mathcal S_\delta\cup\mathcal K_\delta$ in $\Omega_T$.
\end{proof}

\subsection{The admissible class $\mathcal U$}
In what follows,  we say that a function $u$ is  {\em piecewise $C^1$} in $\Omega_T$ and write  $u\in C^1_{pc}(\Omega_T)$ provided that there exists a sequence of disjoint open sets $\{G_j\}_{j=1}^\infty$ in $\Omega_T$  with $|\partial G_j|=0$ such that
\[
u\in C^1(\bar G_j)  \;\; \; \forall\; j\in\mathbf{N}=\{1,2,\cdots\},\quad |\Omega_T\setminus \cup_{j=1}^\infty G_j|=0.
\]
(For our purpose it is also acceptable to allow only \emph{finitely}  many pieces  in this definition.)

\begin{defn} Let $\mu=\|u_t^*\|_{L^\infty(\Omega_T)}+1.$   We  define the {\em admissible class}
\begin{equation}\label{ad-class}
\begin{split}
\mathcal U=\Big \{ &u\in C^1_{pc} \cap W_{u^*}^{1,\infty}(\Omega_T)\;\big| \; \|u_t\|_{L^\infty}<\mu;  \; \exists  v\in C^1_{pc} \cap W_{v^*}^{1,\infty}(\Omega_T;\mathbf R^n)     \\
& \mbox{such that  $\dv v=u$ and $(Du,v_t)\in \mathcal S_\delta\cup \mathcal K_\delta$ a.e.\;in $\Omega_T$}\Big \}.\end{split}
\end{equation}
For each $\epsilon>0$, let $\mathcal U_\epsilon$ be defined by
\[
\begin{split}
\mathcal U_\epsilon= \Big \{ & u\in \mathcal U\;\big| \; \exists \, v\in C^1_{pc} \cap W_{v^*}^{1,\infty}(\Omega_T;\R^n) \mbox{ such that  $\dv v=u$ and}\\&
\mbox{ $(Du,v_t)\in \mathcal S_\delta\cup \mathcal K_\delta$ a.e.\,in $\Omega_T$, and
$\int_{\Omega_T} |v_t-\sigma(Du)|dxdt\le\epsilon|\Omega_T|$} \Big \}. \end{split}
\]
\end{defn}

\begin{remk}\label{rmk-1} Clearly $u^*\in\mathcal U$; so $\mathcal U$ is \emph{nonempty}. Also  $\mathcal U$ is a bounded subset of $W^{1,\infty}_{u^*}(\Omega_T)$ as $\mathcal S_\delta\cup \mathcal K_\delta$ is bounded. Moreover, by Corollary \ref{bound-coro}, for each $u\in\mathcal U$, its corresponding vector function $v$ satisfies $\|v_t\|_{L^\infty(\Omega_T)}\le 1/2$. Finally, note that $m_-(\delta)<|Du^*|\le M$ on some nonempty open subset of $\Omega_T$ and   $A(Du^*)\ne \sigma(Du^*)$ on this set;  hence $u^*$ itself is not a weak solution to (\ref{ib-PM}).
\end{remk}

\section{Density property and Proof of Theorem \ref{thm:main-1}}  

In this final section, we prove the density property of the sets $\mathcal U_\epsilon$ and then complete the proof of Theorem \ref{thm:main-1}.

\subsection{The density property of $\mathcal U$} Let $\mathcal U$ and $\mathcal U_\epsilon$ be as defined in Section 5. We establish  the density property (\ref{density-0}).

\begin{thm} \label{thm-density-1} For each $\epsilon>0$, $\mathcal U_\epsilon$ is dense in $\mathcal U$ under the $L^\infty$-norm.
\end{thm}
\begin{proof}
Let  $u\in \mathcal U$, $\eta>0$. The goal is to show that there exists a function  $\tilde u\in  \mathcal U_\epsilon$ such that $\|\tilde u-u\|_{L^\infty(\Omega_T)}<\eta.$
For clarity, we divide the proof into several steps.

1. Note $\|u_t\|_{L^\infty(\Omega_T)}<\mu-\tau_0$ for some $\tau_0>0$ and there exists a function $v\in C^1_{pc} \cap W_{v^*}^{1,\infty}(\Omega_T;\mathbf R^n)$   such that  $\dv v=u$ and  $(Du,v_t)\in \mathcal S_\delta\cup \mathcal K_\delta$ a.e.\;in $\Omega_T.$
Since both $u$ and $v$ are piecewise $C^1$ in $\Omega_T$,  there exists a sequence of disjoint open sets $\{G_j\}_{j=1}^\infty$ in $\Omega_T$ with  $|\partial G_j|=0\;\forall j\ge 1$ such that
\[
u\in C^1(\bar G_j),
\;v\in C^1(\bar G_j;\R^n)\; \; \forall j\geq 1,\quad |\Omega_T\setminus \cup_{j=1}^\infty G_j|=0.
\] 

2. Let $j\in \mathbf{N}$ be fixed. Note that $(Du(z),v_t(z))\in\bar{\mathcal{S}}_\delta\cup\mathcal{K}_\delta$ for all $z\in G_j$ and that $H_j=\{z\in G_j\;|\;(Du(z),v_t(z))\in\partial\mathcal{S}_\delta\}$ is a (relatively) closed set in $G_j$ with measure zero. So $\tilde G_j=G_j\setminus H_j$ is an open subset of $G_j$ with $|\tilde G_j|=|G_j|$, and $(Du(z),v_t(z))\in\mathcal{S}_\delta\cup\mathcal{K}_\delta$ for all $z\in \tilde G_j$.

3. For each $\tau>0$, let  $\mathcal G_{\tau}=\{(p,\beta)\in  \mathcal S_{\delta} \;|\; |\beta-\sigma(p)|> \tau,\,\mathrm{dist}((p,\beta),\partial\mathcal S_\delta)>\tau)\};$
then $\mathcal G_\tau\subset\subset \mathcal S_{\delta}.$
We can find a $\tau_j>0$ such that
\begin{equation}\label{density-3}
\int_{F_j} |v_t(z)-\sigma(Du(z))|\,dz < \frac{\epsilon}{3\cdot2^{j}}|\Omega_T|,
\end{equation}
where $z=(x,t)$ and $F_j=  \{z\in \tilde G_j\; |\; (Du(z),v_t(z))\notin \mathcal G_{\tau_j}\}.$ To check this, note $F_j=F_{\tau_j}^1\cup F_{\tau_j}^2$, where $F_\tau^1=\{ z\in \tilde G_j\; |\; |v_t(z)-\sigma(Du(z))| \le \tau \}$ and $
F_\tau^2=\{z\in \tilde G_j\;|\; \mathrm{dist}((Du(z),v_t(z)),\partial\mathcal S_\delta)\le \tau\}.$
Clearly, $\int_{F_\tau^1} |v_t-\sigma(Du)|dz \le \tau|\Omega_T|\to 0$ as $\tau\to 0^+$. Since $F_\tau^2$ is decreasing as $\tau\to 0^+$ and $v_t=\sigma(Du)$ on $\cap_{\tau>0} F_\tau^2$,  it follows that $\int_{F_\tau ^2} |v_t-\sigma(Du)|dz \to 0$ as $\tau\to 0^+.$

4. Let $O_j=\tilde G_j\setminus F_j.$ Furthermore, we may require that the number $\tau_j$ be chosen in such a way that either $O_j$ is empty or $O_j$ is a nonempty open set with $|\partial O_j|=0$ (see \cite{Zh}).  Let $J$ be the set of all $j\in\mathbf{N}$ with $O_j\neq\emptyset$. Then for $j\not\in J$, $F_j=\tilde G_j$.

5. We now fix a $j\in J$. Note that $O_j=\{z\in \tilde G_j\; |\; (Du(z),v_t(z))\in \mathcal G_{\tau_j}\}$ and that $\mathcal K_j:=\bar{\mathcal G}_{\tau_j}$ is a compact subset of $\mathcal S_\delta$. Let $\tilde Q\subset\R^n$ be an open box with $\Omega\subset\tilde Q$ and $\tilde I=(0,T)$. Applying Theorem \ref{main-lemma} to box $\tilde Q\times\tilde I,$ $\mathcal K_j$, and $\epsilon'=\frac{\epsilon|\Omega_T|}{12}$, we obtain a constant $\rho_j>0$ that satisfies  the conclusion of the theorem. By the uniform continuity of $\sigma$ on compact subsets of $\R^n$, we can find a $\theta=\theta_{\epsilon,\delta}>0$ such that
\begin{equation}\label{density-5}
|\sigma(p)-\sigma(p')|<\frac{\epsilon}{12}
\end{equation}
if $|p|,\,|p'|\le m_+(\delta)$ and $|p-p'|\le \theta .$
Also by the uniform continuity of $u$, $v$, and their gradients on $\bar G_j$,
there exists a $\nu_j>0$ such that
\begin{equation}\label{density-1}
\begin{array}{c}
  |u(z)-u(z')|+|\nabla u(z)-\nabla u(z')| +|v(z)-v(z')| \\
  +|\nabla v(z)-\nabla v(z')| <\min\{\frac{\rho_{j}}{2},\frac{\epsilon}{12},\theta\} \end{array}
\end{equation}
whenever $z,z'\in \bar G_j$ and  $|z-z'|\le \nu_j.$ We now cover $O_j$ (up to measure zero) by a sequence of disjoint open cubes $\{Q_j^i\times I^i_j\}_{i=1}^\infty$ in $\Omega_T$ whose sides are parallel to the axes with center $z_j^i$ and diameter $l^i_j<\nu_j.$

6. Fix  $i\in\mathbf{N}$ and write  $w=(u,v)$,  $\xi=\begin{pmatrix} p&c\\B & \beta\end{pmatrix}=\nabla w(z_j^i)=\begin{pmatrix} Du(z_j^i) & u_t(z_j^i)\\Dv(z_j^i) & v_t(z_j^i)\end{pmatrix}.$   By the choice of $\rho_j>0$ in Step 5 via Theorem \ref{main-lemma}, since $Q^i_j\times I^i_j\subset\tilde Q\times\tilde I$ and $(p,\beta)\in\mathcal K_j$,  for all sufficiently small $\rho>0$, there exists a function $\omega^i_j=(\varphi^i_j,\psi^i_j)\in C^\infty_c(Q^i_j\times I^i_j;\R^{1+n})$ satisfying

(a) \; $\dv \psi^i_j =0$  in $Q^i_j\times I^i_j$,

(b) \; $(p'+D\varphi^i_j(z), \beta'+ (\psi^i_j)_t(z))\in \mathcal{S}_{\delta}$ for all $z\in Q^i_j\times I^i_j$

\hspace{8mm} and all $|(p',\beta')-(p,\beta)|\leq\rho_{j},$

(c) \; $\|\omega^i_j\|_{L^\infty(Q^i_j\times I^i_j)}<\rho,$

(d) \; $\int_{Q^i_j\times I^i_j} |\beta+(\psi^i_j)_t(z)-\sigma(p+D\varphi^i_j(z))|dz<\epsilon' |Q^i_j\times I^i_j|/|\tilde Q\times \tilde I|,$

(e) \; $\int_{Q^i_j} \varphi^i_j(x,t)dx=0$  for all $t\in I^i_j,$

(f) \; $\|(\varphi^i_j)_t\|_{L^\infty(Q^i_j\times I^i_j)}<\rho.$\\
Here, we let $0<\rho\leq\min\{\tau_0,\frac{\rho_{j}}{2C},\frac{\epsilon}{12C},\eta\}$, where $C_n$ is the constant in Theorem \ref{div-inv} and $C$ is the product of $C_n$ and the sum of the lengths of all sides of $\tilde Q$.
By (e), we can apply Theorem \ref{div-inv} to $\varphi^i_j$ on $Q^i_j\times I^i_j$ to obtain a function $g^i_j=\mathcal R\varphi^i_j\in C^1( \overline{Q^i_j\times I^i_j};\R^n)\cap W^{1,\infty}_0( Q^i_j\times I^i_j;\R^n)$
such that $\dv g^i_j=\varphi^i_j$ in $Q^i_j\times I^i_j$ and
\begin{equation}\label{density-2}
\|(g^i_j)_t\|_{L^\infty(Q^i_j\times I^i_j)}\leq C\|(\varphi^i_j)_t\|_{L^\infty(Q^i_j\times I^i_j)}\leq\frac{\rho_{j}}{2}.
\end{equation}

7. As $|v_t|,\,|\sigma(Du)|\leq1/2\,$ a.e. in $\Omega_T$, we can select a finite index set $\mathcal I\subset J\times \mathbf{N}$ so that
\begin{equation}\label{density-4}
\int_{\bigcup_{(j,i)\in (J\times\mathbf{N}) \setminus \mathcal I}Q^i_j\times I^i_j}|v_t(z)-\sigma(Du(z))|dz\le\frac{\epsilon}{3}|\Omega_T|.
\end{equation}
We finally define
\[
(\tilde u,\tilde v)=(u,v)+\sum_{(j,i)\in\mathcal I}\chi_{Q^i_j\times I^i_j}(\varphi^i_j,\psi^i_j+g^i_j)\;\;\textrm{ in $\Omega_T$.}
\]
As a side remark, note here that only \emph{finitely} many functions $(\varphi^i_j,\psi^i_j+g^i_j) $ are disjointly patched  to the original $(u,v)$ to obtain a new function $(\tilde u,\tilde v)$ towards the goal of the proof. The reason for using only finitely many pieces of gluing is due to the lack of control over the spatial gradients $D(\psi^i_j+g^i_j)$, and overcoming this difficulty is at the heart of  this paper.

8. Let us now check that $\tilde u$ together with $\tilde v$ indeed gives the desired result. By construction, it is clear that $\tilde u\in C^1_{pc}\cap W^{1,\infty}_{u^*}(\Omega_T)$, $\tilde v\in C^1_{pc}\cap W^{1,\infty}_{v^*}(\Omega_T;\R^n).$
By the choice of $\rho$ in (f) as $\rho\le\tau_0$, we have $\|\tilde u_t\|_{L^\infty(\Omega_T)}<\mu.$
Next, let $(j,i)\in\mathcal I,$ and observe that for $z\in Q^i_j\times I^i_j$, with $(p,\beta)=(Du(z^i_j),v_t(z^i_j))\in\mathcal G_{\tau_j}$, since $|z-z_j^i|<l^i_j<\nu_j$, it follows from (\ref{density-1}) and (\ref{density-2}) that
\[
|(Du(z),v_t(z)+(g^i_j)_t(z))-(p,\beta)|\leq\rho_{j},
\]
and so $(D\tilde u(z),\tilde v_t(z))\in\mathcal S_\delta$, by (b) above. From (a) and $\dv g^i_j=\varphi^i_j$, for $z\in Q^i_j\times I^i_j$,
\[
\dv \tilde v(z)=\dv (v+\psi^i_j+g^i_j)(z)=u(z)+0+\varphi^i_j(z)=\tilde u(z).
\]
Therefore, $\tilde u\in\mathcal U$.
Next, observe
\[
\int_{\Omega_T}|\tilde v_t-\sigma(D\tilde u)|dz= \int_{\cup_{j\in\mathbf{N}}F_j}|v_t-\sigma(Du)|dz
\]
\[
+\int_{\cup_{(j,i)\in (J\times\mathbf{N}) \setminus \mathcal I}Q^i_j\times I^i_j}|v_t-\sigma(Du)|dz+\int_{\cup_{(j,i)\in  \mathcal I}Q^i_j\times I^i_j}|\tilde v_t-\sigma(D\tilde u)|dz
\]
\[
= I_1+I_2+I_3.
\]
 From (\ref{density-3}) and (\ref{density-4}), we have $I_1+I_2\leq\frac{2\epsilon}{3}|\Omega_T|$. Note also that for $(j,i)\in\mathcal I$ and $z\in Q^i_j\times I^i_j,$ from (\ref{density-1}), (\ref{density-2}), and (f),
\[
|\tilde v_t(z)-\sigma(D\tilde u(z))|=|v_t(z)+(\psi^i_j)_t(z)+(g^i_j)_t(z)-\sigma(Du(z)+D\varphi^i_j(z))|
\]
\[
\leq |v_t(z)-v_t(z^i_j)|+|v_t(z^i_j)+(\psi^i_j)_t(z)-\sigma(Du(z^i_j)+D\varphi^i_j(z))|
\]
\[
+|(g^i_j)_t(z)|+|\sigma(Du(z^i_j)+D\varphi^i_j(z))-\sigma(Du(z)+D\varphi^i_j(z))|
\]
\[
\le\frac{\epsilon}{6}+|v_t(z^i_j)+(\psi^i_j)_t(z)-\sigma(Du(z^i_j)+D\varphi^i_j(z))|
\]
\[
+|\sigma(Du(z^i_j)+D\varphi^i_j(z))-\sigma(Du(z)+D\varphi^i_j(z))|.
\]
Here, as $(D\tilde u(z),\tilde v_t(z))\in\mathcal S_\delta$, we have $|Du(z)+D\varphi^i_j(z)|=|D\tilde u(z)|\le m_+(\delta)$, and by (\ref{density-1}), $|Du(z^i_j)-Du(z)|<\theta$. From (\ref{density-5}) we now have
\[
|\sigma(Du(z^i_j)+D\varphi^i_j(z))-\sigma(Du(z)+D\varphi^i_j(z))|<\frac{\epsilon}{12}.
\]
Integrating the above inequality over $Q^i_j\times I^i_j,$ we thus obtain from (d) that
\[
\int_{Q^i_j\times I^i_j}|\tilde v_t(z)-\sigma(D\tilde u(z))|dz\le\frac{\epsilon}{4}|Q^i_j\times I^i_j|+\frac{\epsilon|\Omega_T|}{12}\frac{|Q^i_j\times I^i_j|}{|\tilde Q\times \tilde I|}\]
\[
\leq\frac{\epsilon}{3}|Q^i_j\times I^i_j|,
\]
which yields that $I_3\leq\frac{\epsilon}{3}|\Omega_T|$. Hence $I_1+I_2+I_3\le\epsilon|\Omega_T|$, and so $\tilde u\in\mathcal U_\epsilon.$
Finally, from (c) with $\rho\le \eta$ and the definition of $\tilde u$, we have 
\[
\|\tilde u-u\|_{L^\infty(\Omega_T)}<\eta.
\]
This completes the proof.
\end{proof}

\subsection{Completion of Proof of Theorem \ref{thm:main-1}}   The existence of infinitely many weak solutions to problem (\ref{ib-PM}) readily follows from a combination of Remark \ref{rmk-1}, Theorem \ref{thm-density-1},  and Theorem \ref{thm:main}. 
 
To prove the last statement of Theorem \ref{thm:main-1},  assume $M=\|Du_0\|_{L^\infty(\Omega)}\ge 1$ and $\lambda>0$.  Recall that  in this case $\delta=\frac{M+\lambda}{1+(M+\lambda)^2}.$ By the definition of $\mathcal{U}$ and Corollary \ref{bound-coro}, we  have $|Du|\leq m_+(\delta)=M+\lambda$ a.e. in $\Omega_T$ for all $u\in\mathcal U.$ On the other hand, following Step 3 in the proof of Theorem \ref{thm:main}, every weak solution $u\in\mathcal G$ is the $L^{\infty}$-limit of some sequence $u_j\in\mathcal U$ such that $Du_j\to Du$ a.e. in $\Omega_T$.  So these weak solutions $u\in\mathcal G$ must satisfy $\|Du\|_{L^\infty(\Omega_T)}\le M+\lambda.$

\end{document}